\documentclass[12pt]{article}
\usepackage[dvips]{color}
\usepackage{geometry}
\usepackage{graphicx}
\usepackage{amsmath,amsthm}
\usepackage{amssymb}
\usepackage{multicol}
\usepackage{ulem}
\usepackage{amsfonts}
\usepackage{mathrsfs}
\usepackage[pdfstartview=FitH]{hyperref}
 \newtheorem{thm}{Theorem}[section]
 \newtheorem{cor}[thm]{Corollary}
 \newtheorem{lem}[thm]{Lemma}
 \newtheorem{prop}[thm]{Proposition}
 
 \newtheorem{dfn}[thm]{Definition}
 \theoremstyle{definition}
  
 \newtheorem{exmp}{Example}
 
 \newtheorem{rem}{Remark}
  \DeclareMathAlphabet{\mathsfsl}{OT1}{cmss}{m}{sl}

\newcommand{\IM}{\mathrm{Im}}
  \newcommand{\RE}{\mathrm{Re}}
  
  \newcommand{\DR}{\mathbb{D}}
  \newcommand{\FH}{\mathfrak{H}}
 
 \newcommand{\Rnum}{\mathbb{R}}
 \newcommand{\Cnum}{\mathbb{C}}
 
 \newcommand{\Nnum}{\mathbb{N}}
 \newcommand{\mi}{\mathrm{i}}
 \newcommand{\dif}{\mathrm{d}}
  \newcommand{\diag}{\mathrm{diag}}
 \newcommand{\tensor}[1]{\mathsf{#1}}
 \newcommand{\abs}[1]{\left\vert#1\right\vert}
 \newcommand{\set}[1]{\left\{#1\right\}}
 \newcommand{\norm}[1]{\left\Vert#1\right\Vert}
 \newcommand{\innp}[1]{\langle {#1}\rangle}

\allowdisplaybreaks

\title{On the fourth moment theorem for the complex multiple Wiener-It\^{o} integrals}
\author{\rm\small
\noindent  Yong CHEN\\
\noindent \footnotesize School of Mathematics and Computing Science, Hunan
University of Science and Technology,\\
\noindent \footnotesize Xiangtan, Hunan, {\rm 411201},
P.R.China. zhishi@pku.org.cn\\
\rm\small \noindent Yong LIU\\
\noindent \footnotesize LMAM, School of Mathematical Sciences, Peking University,\\
\noindent \footnotesize Beijing,  {\rm 100871}, P. R. China.
liuyong@math.pku.edu.cn (Corresponding author)\\
}
\date{}
\begin{document}
\maketitle
\maketitle \noindent {\bf Abstract } \\
In this paper, a product formula of Hermite polynomials is given and then the relation between the real Wiener-It\^{o} chaos and the complex Wiener-It\^{o} chaos (or: multiple integrals) is shown. By this relation and the known multivariate extension of the fourth moment theorem for the real multiple integrals, the fourth moment theorem (or say: the Nualart-Peccati criterion) for the complex  Wiener-It\^{o} multiple integrals is obtained. 
\maketitle
\tableofcontents
\begin{center}
   {\bf Notations}
   \begin{eqnarray*}
\FH  &:&\text{a real separable Hilbert space} \\
\FH^{\odot m} &:& \text{the $m$ times symmetric tensor product of $\FH$}\\
 \FH\oplus\FH &:& \text{the Hilbert space direct sum}\\
\FH_{\Cnum} &:& \text{the complexification of $\FH$}\\
\FH_{\Cnum}^{\odot m} &:&\text{the $m$ times symmetric tensor product of $\FH_{\Cnum}$}\\
X,Y &:&\text{the real Gaussian isonormal process over $\FH$} \\
W &: & \text{the real Gaussian isonormal process over $\FH\oplus \FH$}\\
X_{\Cnum},Y_{\Cnum} &:&\text{the complexicfation of $X,\,Y$}\\
Z &:& \text{the complex Gaussian isonormal process over $\FH_{\Cnum}$}\\
\mathcal{H}_n(X),\,\mathcal{H}_n(Y),\,\mathcal{H}_n(W)&:&\text{the $n$-th Wiener-It\^{o} chaos of $X,Y,W$}\\
\mathcal{H}^{\Cnum}_n(X),\,\mathcal{H}^{\Cnum}_n(W)&:&\text{the complexification of $\mathcal{H}_n(X),\,\mathcal{H}_{n}(W)$} \\
\mathscr{H}_{m,n}(Z) &:& \text{the $(m,n)$-th complex Wiener-It\^{o} chaos of $Z$}\\
{\rm symm }(f\otimes g) &:& \text{symmetrizing tensor product of $f$ and $g$ }
   \end{eqnarray*}
\end{center}

\section{Introduction}
In a seminal paper \cite{nup}, Nualart and Peccati showed that the convergence in distribution of a normalized sequence of real multiple Wiener-It\^{o} integrals towards a standard Gaussian law is equivalent to convergence of just the fourth moment to 3, which is
called the Nualart-Peccati criterion or the fourth moment theorem. Shortly afterwards, Peccati and Tudor \cite{pt} gave a
multivariate extension of this characterization. After the publication of the two beautiful papers, there are already several proofs of the criterion such as \cite{acp,kut,led,nour,np0,npr,nourorti}. Especially, Nualart and Ortiz-Latorre \cite{nourorti} presented a crucial methodological breakthrough, linking the criterion to Mallliavin operators, Nourdin and Peccati \cite{np0} established the combination of Stein's method and Malliavin calculus, and the recent papers \cite{led,acp} by Ledoux, Azoodeh, Campese and Poly were from the point of view of spectral theory of general Markov diffusion generators. For details, please refer to the monograph  \cite{np} written by Nourdin and Peccati. In addition, Nourdin and Peccati \cite{np1} showed that the convergence in distribution of a sequence of real multiple Wiener-It\^{o} integrals towards a centered $\chi^2$ law is equivalent to convergence of just the fourth moment and the third moment, and the multivariate extension of this theorem was shown by Nourdin and Rosi\'{n}ski recently \cite{nr}. Hu, Lu, Nourdin, Nualart and Poly \cite{hln,nournp,nourp} strengthened the convergence in law to the uniform convergence of the densities and the total variation convergence (which is equivalent to the $L^1(\Rnum^d)$ convergence of the densities) respectively.

Since both the real multiple Wiener-It\^{o} integrals and the complex multiple Wiener-It\^{o} integrals were established by K. It\^{o} almost at the same time in 1950s \cite{ito2,ito}, the question naturally arises if the Nualart-Peccati criterion is still valid for the complex multiple Wiener-It\^{o} integrals. The principle aim of this paper is to give a positive answer to the above-presented question. Our main results are the following two Nualart-Peccati criterions in abstract complex Wiener-It\^{o} chaos (see Definition~\ref{df27}).

For the rest of the paper, we shall denote by $\zeta\sim \mathcal{CN}(0,\sigma^2)$ a symmetric complex Gaussian variable, i.e., $\zeta=\xi_1+\mi\xi_2$ with $\xi_i\sim \mathcal{N}(0,\frac12 \sigma^2)$ and independent.
\begin{thm}\label{th1}
 Consider a sequence of random variable $F_k$ being the fixed $(m,n)$-th complex Wiener-It\^{o} multiple integrals, $m+n\ge 2$ and suppose that $E[\abs{F_k}^2]\to \sigma^2$ as $k\to \infty$, where $\abs{\cdot}$ is the absolute value (or modulus) of a complex number.
  \begin{itemize}
    \item[\textup{1)}]  If $m\neq n$, as $k\to \infty$, the following two assertions are equivalent:
  \begin{itemize}
    \item[\textup{(i)}] The sequence $(F_k)$ converges in distribution to $\zeta\sim \mathcal{CN}(0,\sigma^2)$;
    \item[\textup{(ii)}] $E[\abs{F_k}^4]\to 2\sigma^4$. 
  \end{itemize}
   \item[\textup{2)}]  If $m= n$ and $E[F_k^2]\to \sigma^2(a+\mi b)$ where $a,b\in \Rnum$ such that $a^2+b^2< 1$, that is to say the matrix $\tensor{C}=\begin{bmatrix} 1+a & b\\ b & 1-a  \end{bmatrix}$ is positive definite, the following two assertions are equivalent:
  \begin{itemize}
    \item[\textup{(i)}] The sequence $(\RE F_k,\,\IM  F_k)$ converges in distribution to a jointly normal law with the covariance $\frac{\sigma^2}{2}\tensor{C}$,
    \item[\textup{(ii)}] $E[\abs{F_k}^4]\to(a^2+b^2+2)\sigma^4$.
 \end{itemize}
  \item[\textup{3)}]  If $m= n$ and $E[F_k^2]\to \sigma^2(a+\mi b)$ where $a,b\in \Rnum$ such that $a^2+b^2=1$, i.e., the matrix $\tensor{C}$ is degenerated, the following three assertions are equivalent: \begin{itemize}
    \item[\textup{(i)}] The sequence $(\RE F_k,\,\IM  F_k)$ converges in distribution to a jointly normal law with the covariance $\frac{\sigma^2}{2}\tensor{C}$,
    \item[\textup{(ii)}] $E[\abs{F_k}^4]\to 3\sigma^4$,
    \item[\textup{(iii)}] $E[{F_k}^4]\to 3(a+\mi b)^2\sigma^4$.
 \end{itemize}
 \end{itemize}
\end{thm}
\begin{rem}
   Especially, in \textup{3)}, if $a=\pm 1$, i.e. the sequence $(F_k)$ (or the sequence $(\mi F_k)$ ) is real, then as $k\to \infty$, it converges in distribution to $\mathcal{N}(0,\sigma^2)$ if and only if $E[{F_k}^4]\to 3\sigma^4$, which is just the original Nualart-Peccati criterion.
\end{rem}
\begin{thm}\label{th2}Let $\xi(\alpha_1,\alpha_2)=G_1(\alpha_{1})+\mi G_2({\alpha}_{2})$ be a complex random variable such that $G_i(\alpha_i),\,i=1,2,$ being independent variables having centered $\chi^2$ distributions with $\alpha_i$ degree of freedom respectively.
 Consider a sequence of random variable $F_k$ belonging to the $(m,n)$-th complex Wiener-It\^{o} chaos, $m+n\ge 2$ being an even number and suppose that $E[\abs{F_k}^2]\to \sigma^2$ as $k\to \infty$.
  \begin{itemize}
    \item[\textup{1)}]  If $m\neq n$, as $k\to \infty$, the following two assertions are equivalent:
  \begin{itemize}
    \item[\textup{(i)}] The sequence $(F_k)$ converges in distribution to $\xi(\sigma^2/2,\,\sigma^2/2) $;
    \item[\textup{(ii)}] $E[F_k^3+3|F_k|^2\bar{F}_k]\to 8(1-\mi)\sigma^2$ and $E[\abs{F_k}^4]\to 2\sigma^4+24\sigma^2$.
        \end{itemize}
    \item[\textup{2)}]  If $m= n$ and $E[F_k^2]\to \sigma^2(a+\mi b)$ where $a,b\in \Rnum$ such that $a^2+b^2< 1$, as $k\to \infty$, the following two assertions are equivalent:
          \begin{itemize}
    \item[\textup{(i)}] The sequence $(F_k)$ converges in distribution to $\xi(\frac{1+a}{2}\sigma^2,\,\frac{1-a}{2}\sigma^2) $;
    \item[\textup{(ii)}] $E[F_k^3+3|F_k|^2\bar{F}_k]\to 8[1+a-\mi(1-a)]\sigma^2$ and $E[\abs{F_k}^4]\to (2+a^2)\sigma^4+24\sigma^2$.
        \end{itemize}
   \end{itemize}
\end{thm}
\begin{rem}
   In the above theorem, when $m+n$ is an odd integer, there does not exist any $(F_k)$ with bounded variances converging in distribution to $\xi(\alpha_1,\alpha_2)$ as $k\to \infty$ \cite{np1}.
\end{rem}
\begin{rem}
   It follows from Theorem 5.2 and Corollary 5.5 in \cite{nournp} by Nourdin, Nualart and Poly that we can strengthen the convergence in law of Theorem~\ref{th1} (when $\tensor{C}$ is positive definite) and Theorem~\ref{th2} to the convergence in total variation. Denote by $ \Gamma(F_k)$ the Malliavin matrix of $F_k=(\RE F_k,\,\IM  F_k)$. As a consequence of Theorem 5.2 and Corollary 5.5 in \cite{nournp}, we deduce that $E[\det \Gamma(F_k)]$ is bounded away form zero. Then each $F_k$ admits a density and the above convergence in total variation is equivalent to the convergence of the densities in $L^1(\Rnum^2)$.
\end{rem}

The key idea of the proof of the main results above is based on Theorem~\ref{co2} and Theorem~\ref{pp2} in section 3. Essentially, Theorem~\ref{pp2} means that both the real part and the imaginary part of a complex multiple integral can be represented by real multiple integrals respectively. Therefore, we can utilize the known multivariate extension of the fourth moment theorem for the real multiple integrals. According to our knowledge, Theorem~\ref{co2} and Theorem~\ref{pp2} are new and the proof is non-trivial.

The rest of the paper is organized as follows. In section~\ref{sec2}, we give some properties of the connection between the real and the complex Hermite polynomials by the view of complex (real) Herimite polynomials being the eigenfunctions of complex (real) Ornstein-Uhlenbeck operator. These properties play an important role of the proofs of the main results. In section~\ref{sec222}, similar to the standard definition of real isonormal Gaussian process (please refer to \cite{jan,np, Nua} and references therein), we define the complex isonormal Gaussian processes and the isometric mapping onto the complex Wiener-It\^{o} chaos. 
In section~\ref{sec23}, we obtain a product formula of Hermite polynomials (see Theorem~\ref{mth00}) and then show the relation between the real Wiener-It\^{o} multiple integrals and the complex Wiener-It\^{o} multiple integrals (see Theorem~\ref{co2} and Theorem~\ref{pp2}). In section~\ref{exp2}, we revisit the classical theory of It\^{o}'s complex multiple integrals and express the abstract Theorem~\ref{mth00}-\ref{co2} in the classical Wiener-It\^{o}'s multiple integrals. The proofs of the main theorems of the paper (Theorem~\ref{th1}-\ref{th2} and Theorem~\ref{co2}-\ref{pp2}) are presented in section~\ref{sec4}. In section~\ref{sec30}, we generalized Theorem~\ref{th1}-\ref{th2} slightly to the case of finite orthogonal sum of Wiener-It\^{o} chaos and the multivariate case.

There exists an alternative definition and representation of the complex multiple Wiener-It\^{o} integrals by Malliavin calculus. For the completion of the theory, we summarize concisely those facts as an appendix finially.
\section{Preliminaries}\label{sec22}
\subsection{Some properties of complex Hermite polynomials }\label{sec2}

Consider a 1-dimensional complex-valued Ornstein-Uhlenbeck process \cite{cl}
\begin{equation}\label{cp}
  \dif C_t=-e^{\mi \theta} C_t\dif t+ \sqrt{\rho\cos\theta}\dif \zeta_t
\end{equation}
where $C_t=C_1(t)+\mi C_2(t)$,  $\theta\in(-\frac{\pi}{2},\,\frac{\pi}{2}),\,\rho>0$, and $\zeta_t$ is a complex Brownian motion. It is clear that this complex-valued process can be represented by the 2-dimensional nonsymmetric (when $\theta\neq 0$) Ornstein-Uhlenbeck process
\begin{equation*}
  \begin{bmatrix}
\dif C_1(t)\\ \dif C_2(t) \end{bmatrix} =\begin{bmatrix} -\cos \theta & \sin \theta\\ - \sin \theta & -\cos \theta \end{bmatrix}
\begin{bmatrix} C_1(t)\\ C_2(t) \end{bmatrix} \dif t + \sqrt{\rho\cos\theta}
\begin{bmatrix} \dif B_1(t)\\ \dif B_2(t) \end{bmatrix}
\end{equation*}
Its generator is
\begin{align}\label{ou.op}
A_{\theta} & =\frac{\rho\cos\theta}{2}(\frac{\partial^2}{\partial x^2}+\frac{\partial^2}{\partial y^2})+(- x\cos\theta+ y\sin\theta)\frac{\partial}{\partial x}-( x\sin\theta+y\cos\theta)\frac{\partial}{\partial y}\nonumber\\
&=2\rho\cos\theta \frac{\partial^2}{\partial z\partial \bar{z}}-e^{\mi\theta} z \frac{\partial}{\partial z}-e^{-\mi\theta}\bar{z} \frac{\partial}{\partial \bar{z}},
\end{align}
which is nonsymmetric (when $\theta\neq 0$) but normal, where
$\frac{\partial f}{\partial z}=\frac12 (\frac{\partial f}{\partial x}-\mi \frac{\partial f}{\partial y}), \frac{\partial f}{\partial \bar{z}}=\frac12 (\frac{\partial f}{\partial x}+\mi \frac{\partial f}{\partial y})$ are
the formal derivative of $f$ at point $z=x+\mi y$ with $x,y\in\Rnum$. We call $\partial:=\frac{\partial }{\partial z}$ and $ \bar{\partial}:=\frac{\partial }{\partial \bar{z}}$ the complex annihilation operators. In \cite[Theorem 2.7]{cl},  the authors show that for any $\theta\in (-\frac{\pi}{2},\frac{\pi}{2})$ and $\rho>0$, $A_{\theta}$ satisfies that
\begin{equation}\label{atheta}
   A_{\theta} J_{m,n}(z,\rho)=-[(m+n)\cos \theta +\mi (m-n)\sin\theta]J_{m,n}(z,\rho),
\end{equation}
where $ J_{m,n}(z,\rho)$ is the so-called complex Hermite polynomials (or say: Hermite-Laguerre-It\^{o} polynomials) given by
\begin{equation}\label{itldefn}
  \begin{array}{ll}
  J_{0,0}(z,\rho)&=1,\\
  J_{m,n}(z,\rho)&=\rho^{m+n}(\partial^*)^m(\bar{\partial}^*)^n 1,\quad m,n\in \Nnum,
    \end{array}
\end{equation}
where $ (\partial^*\phi)(z)=-\frac{\partial}{\partial \bar{z}}\phi(z)+\frac{z }{\rho}\phi(z),\quad (\bar{\partial}^*\phi)(z)=-\frac{\partial}{\partial {z}}\phi(z)+\frac{\bar{z}}{\rho }\phi(z)$ for $\phi\in C^1_0(\Rnum^2)$ are the adjoint of the operators $\partial,\,\bar{\partial} $ respectively (the complex creation operator).
$$\set{((m!n!\rho^{m+n})^{-\frac12} J_{m,n}(z,\rho):\,m,n\in \Nnum}$$ is a complete orthonormal system \cite{cl,ito} of $L^2_\Cnum(\Cnum,\,\nu)$ with $\dif\nu=\frac{1}{\pi \rho}e^{-\frac{x^2+y^2}{\rho}}\dif x\dif y$. If $\rho=2$, we will often write $J_{m,n}(z)$ instead of $J_{m,n}(z,\rho)$.

The real Hermite polynomials $H_n$ are defined by the formula\footnote{Note that $H_n(x)=\frac{(-1)^n}{n!}  e^{x^2/2}\frac{\dif^n}{\dif x^n}e^{-x^2/2}$ in \cite{Nua,nup, shg} and $H_n(x)=\frac{(-1)^n}{\sqrt{n!}} e^{x^2/2}\frac{\dif^n}{\dif x^n}e^{-x^2/2}$ in \cite{str}, here we use the definition in \cite{cl,guo,np}.} \cite[p157]{guo}  $$H_n(x)=(-1)^n e^{x^2/2}\frac{\dif^n}{\dif x^n}e^{-x^2/2},\,n=1,2,\dots.$$
The following property gives the fundamental relation between the real and the complex Hermite polynomials \cite[Corollary 2.8]{cl} by the authors, which plays the important role of the proofs of the main theorems of the present paper (see Theorem~\ref{mth00}-\ref{co2} below ).
\begin{prop}\label{2dim2}
Let $z=x + \mi y$ with $x,y\in \Rnum$. Then the real and the complex Hermite polynomials satisfy that
\begin{equation}\label{h2j}
    \begin{array}{ll}
      J_{m,l-m}(z) = \sum\limits_{k=0}^{l}{\mi^{l-k}}\sum_{r+s=k}{m \choose r}{l-m \choose s}(-1)^{l-m-s} H_k(x)H_{l-k}(y),\\
        H_k(x)H_{l-k}(y) = \frac{\mi^{l-k}}{2^{l}}\sum\limits_{m=0}^l \sum_{r+s=m}{k \choose r}{l-k \choose s}(-1)^{s} J_{m,l-m}(z).
    \end{array}
 \end{equation}
 Thus, both the class $\set{ J_{k,l}(z):\,k+l=n}$ and the class $\set{ H_k(x)H_{l}(y):\,k+l=n }$ generate the same linear subspace of $L^2_{\Cnum}(\Cnum,\nu)$.
\end{prop}

Equality (\ref{atheta}) means that $J_{m,n}(z,\rho)$ is the eigenfunctions of $A_{\theta}$ for all $\theta\in(-\frac{\pi}{2},\frac{\pi}{2})$, and the only difference is the eigenvalue. Especially when $\theta=0$, the normal Ornstein-Uhlenbeck operator $A_{\theta}$ degenerates to a symmetric operator $A_0=\frac{\rho}{2}(\frac{\partial^2}{\partial x^2}+\frac{\partial^2}{\partial y^2})-x\frac{\partial}{\partial x}-y\frac{\partial}{\partial y}$, and it follows from (\ref{atheta}) that both the real part and the imaginary part of $J_{m,n}(z,\rho)$ are the eigenfunctions of $A_0$ with respect to the same eigenvalue $-(m+n)$. Moreover, since \cite[Proposition A.6]{cl}
\begin{align}
   \overline{J_{m,n}(z,\rho)}=J_{n,m}(z,\rho),\label{jbar}
\end{align}
we have that when $m\neq n$, $E_{\nu}[J_{m,n}(z,\rho)^2]=E_{\nu}[J_{m,n}(z,\rho)\overline{J_{n,m}(z,\rho)}]=0 $ which implies that the real part and the imaginary part of $J_{m,n}(z,\rho)$ are orthogonal and having the same norm in $L^2_\Cnum(\Cnum,\,\nu)$. We conclude it as a proposition.
\begin{prop}\label{2dims}
   Let $J_{m,n}(z,\rho)=f+\mi g$, then \begin{align*}
      A_0f&=-(m+n)f,\quad A_0g=-(m+n)g.  \end{align*}
      If $m\neq n$ then $f,\,g$ satisfy
   \begin{align*}
      \norm{f}_{L^2_\Cnum(\nu)}&=\norm{g}_{L^2_\Cnum(\nu)},\quad E_{\nu}[fg]=0.
   \end{align*}
\end{prop}

The above basic properties give a heuristic answer to the problem of the relation between real multiple integrals and complex multiple integrals. In fact, an infinite dimensional version of Proposition~\ref{2dims} is given by Theorem~\ref{pp2} below.

\subsection{Complex Gaussian isonormal process and complex Wiener-It\^{o} chaos}\label{sec222}
Before we describe the formulation of Wiener-It\^{o} chaos decomposition theorem for complex Gaussian isonormal process, let us recall the corresponding theory of real Gaussian isonormal process.

The standard approach to define the Wiener-It\^{o} chaos is using Hermite polynomials (please refer to \cite{ito2}, \cite[Definiton~1.1.1]{Nua} and \cite[Definiton~2.2.3]{np}, or see Definition~\ref{nott01} below), and an alternative standard (but equivalent by Proposiiton\ref{rmw}) way is using general polynomial vector spaces  (please refer to \cite[Definition 2.1]{jan}). Here we adopt the former.
\begin{dfn}\label{nott01}
For a fixed real separable Hilbert space $\FH$, an {\it isonormal Gaussian process} over $\FH$, $X=\set{X(h):\,h\in \FH}$ means that $X$ is a centered Gaussian family defined on some probability space $(\Omega, \mathcal{F}_0,P)$ and such that $E[X(g){X(h)}]=\innp{g,h}_{\FH}$ for every $g,h\in \FH$. If $\set{e_i:,\,i\ge 1}$ is a countable orthonormal basis of $\FH$ and $\set{\xi_i}$ is a sequence of i.i.d. standard normal random variables, then $X$ is uniquely determined in the sense of law by
\begin{equation}\label{xh}
   X(h)=\sum^{\infty}_{i=1}\innp{h,\,e_i}_{\FH}\xi_i.
\end{equation}
The $n$-th Wiener-It\^{o} chaos $\mathcal{H}_{n}(X)$ of $X$ is the closed linear subspace of the real $L^{2}(\Omega)$  generated by the random variable of the type $\set{H_n(X(h)),\,h\in \FH,\norm{h}=1}$ where $H_n$ is the $n$-th Hermite polynomial.
\end{dfn}
For a sequence $\mathbf{m}=\set{m_k}_{k=1}^{\infty}$ of nonnegative integrals with finite sum, we set $\abs{\mathbf{m}}=\sum_{k=1}^{\infty} m_k$ and $\mathbf{m}!=\prod_{k=1}^{\infty}  m_k!$ and define a Fourier-Hermite polynomial \cite{Nua,shg}
\begin{equation}\label{bfh}
   \mathbf{H}_{\mathbf{m}}:=\frac{1}{\sqrt{\mathbf{m}!}}\prod_{k=1}^{\infty} H_{m_k}(X(e_k)).
\end{equation}
Set $\xi_k=X(e_k)$ and $\xi=\set{\xi_k:k=1,2,\dots}$ and use the notation $\xi^{\mathbf{m}}:=\prod_{k=1}^{\infty}\xi_k^{m_k}$, then the right hand side of (\ref{bfh}) is exactly the wick product $:\xi^{\mathbf{m}}:$ (please refer to \cite[Theorem 3.15]{jan}).
The next two results are well-known for the real isonormal Gaussian process \cite{np,Nua}.
\begin{prop}{\bf(Wiener-It\^{o} chaos decomposition) }\begin{itemize}
   \item[\textup{i)}] The linear space generated by the class $\set{H_n(X(h)):\,n\ge 0, \norm{h}_{\FH}=1}$ is dense in $L^q(\Omega)$ for every $q\ge 1$.
   \item[\textup{ii)}] The space $L^2(\Omega,\sigma(X),P)$ can be decomposed into the infinite orthogonal sum of the subspace $\mathcal{H}_n(X)$, i.e., $L^2(\Omega)=\bigoplus_{n=0}^{\infty}\mathcal{H}_n(X)$.
\end{itemize}
\end{prop}
\begin{prop}\label{rmw}{\bf (Real multiple Wiener-It\^{o} integral)}\begin{itemize}
   \item[\textup{i)}] For any $m\ge 1$ the random variables $\set{\mathbf{H}_{\mathbf{m}}:\,\abs{\mathbf{m}}=m}$ form a complete orthonormal system in $\mathcal{H}_m(X)$.
   \item[\textup{ii)}] The linear mapping $\mathcal{I}_m(\mathrm{symm}(\bigotimes_{k=0}^{\infty}e_k^{\otimes m_i}))=\sqrt{\mathbf{m}!}\mathbf{H}_{\mathbf{m}}$ provides an isometry from the tensor product $\FH^{\odot m}$, equipped with the norm $\sqrt{m!}\norm{\cdot}_{\FH^{\otimes m}}$, onto $\mathcal{H}_m(X)$. For any $f\in \FH^{\odot m}$, $\mathcal{I}_{m}(f)$ is called the real multiple Wiener-It\^{o} integral of $f$ with respect to $X$ (please refer to \cite{np1} or \cite{nup}).
\end{itemize}
\end{prop}


Now, we turn to the definitions of complex isonormal process, which stem from Ito's work on complex multiple integrals essentially \cite{ito}. But we narrate them in the terminology appeared in \cite{jan} and \cite{Nua}.

Let $\set{ \eta_i:\,i\ge 1}$ be an independent copy of $\set{\xi_i}$ on some probability space $(\Omega, \mathcal{F},P)$, then $Y=\set{Y(h):\,h\in \FH}$ satisfying
\begin{align}\label{yh}
   Y(h)=\sum^{\infty}_{i=1}\innp{h,\,e_i}_{\FH}\eta_i
  \end{align}
  is an independent copy of the isonormal Gaussian process $X$ over $\FH$. Then we complexify $\FH$ and $L^{2}(\Omega)$ in the usual way and denote by $\FH_{\Cnum}$ and $L_{\Cnum}^{2}(\Omega)$ respectively.
Suppose that $\FH_{\Cnum}\ni \mathfrak{h}=f+\mi g$ with $f,g\in \FH$, we write
\begin{equation}\label{xh2}
   X_{\Cnum}(\mathfrak{h}):=X(f)+\mi X(g)=\sum_{i=1}^{\infty} \innp{\mathfrak{h},\,e_i}_{\FH_{\Cnum}}\xi_i,
\end{equation}
 which satisfies $E[X_{\Cnum}(\mathfrak{h})\overline{X_{\Cnum}(\mathfrak{h}_1)}]=\innp{\mathfrak{h},\,\mathfrak{h}_1}_{\FH_{\Cnum}}$, where $\mathfrak{h}_1\in\FH_{\Cnum}$.
The complexification of $\mathcal{H}_{n}(X)$ is given by
\begin{align}\label{hcx}
   \mathcal{H}^{\Cnum}_{n}(X):=\mathcal{H}_{n}(X)+\mi\mathcal{H}_{n}(X)=\set{F+\mi G:\,F,G\in \mathcal{H}_{n}(X)},
\end{align} which is the closed linear subspace of $L_{\Cnum}^{2}(\Omega)$ generated by the random variable of the type $\set{H_n(X(h)),\,h\in \FH,\norm{h}=1}$.
Clearly, $\set{e_i:i\ge 1}$ is still the basis of the complex Hilbert space $ \mathfrak{H}_{\Cnum}$.

\begin{dfn}\label{nott02}
 Let $\set{\zeta_i=\xi_i+\mi \eta_i:\,i\ge 1}$ be a sequence of i.i.d. symmetric complex normal random variables with variance $2$ on the probability space $(\Omega, \mathcal{F},P)$. Then
\begin{align}\label{zh}
   Z(\mathfrak{h})&=\frac{X_{\Cnum}(\mathfrak{h})+\mi Y_{\Cnum}(\mathfrak{h})}{\sqrt2}= \frac{1}{\sqrt2}\sum^{\infty}_{i=1}\innp{\mathfrak{h},\,e_i}_{\FH_{\Cnum}}\zeta_i,\quad \mathfrak{h}\in \FH_{\Cnum}
\end{align}
is called a {\it complex isonormal Gaussian process} over $ \FH_{\Cnum}$, which is a centered symmetric complex Gaussian family and such that
\begin{align*}
   E[ Z(\mathfrak{h})^2]=0,\quad E[Z(\mathfrak{g})\overline{Z(\mathfrak{h})}]=\innp{\mathfrak{g},\mathfrak{h}}_{\FH_{\Cnum}},\quad \forall \mathfrak{g},\mathfrak{h}\in \FH_{\Cnum}.
\end{align*}
\end{dfn}
\begin{rem}
Eq.(\ref{zh}) is exactly Janson's idea of the isometric complex Gaussian Hilbert space (see Example 1.9, Theorem 1.23 of \cite{jan} and \cite[p15]{jan}).
\end{rem}

\begin{dfn}\label{df27}
   For each $m,n\ge 0$, we write $\mathscr{H}_{m,n}(Z)$ to indicate the closed linear subspace of $L_{\Cnum}^2(\Omega)$ generated by the random variables of the type
   $J_{m,n}(Z(\mathfrak{h})), \mathfrak{h}\in \mathfrak{H}_{\Cnum},\norm{\mathfrak{h}}_{\mathfrak{H}_{\Cnum}}=\sqrt2$ where $J_{m,n}(z)$ is the complex Hermite polynomials given by (\ref{itldefn}). The space $\mathscr{H}_{m,n}(Z)$ is called the Wiener-It\^{o} chaos of degree of $(m,n)$ of  $Z$(or say: $(m,n)$-th Wiener-It\^{o} chaos of $Z$.)
\end{dfn}
\begin{rem}
Using the polynomial vector spaces, Janson defines complex Wiener-It\^{o} chaos $H_{\Cnum}^{:n:}$. Due to Proposition~\ref{ppp1}, it equals to $$\bigoplus_{k+l=n}\mathscr{H}_{k,l}(Z)$$
in our notation, but he does not decompose $H_{\Cnum}^{:n:}$ into the orthogonal direct sum of $\mathscr{H}_{k,l}(Z)$. 
\end{rem}
\begin{dfn}\label{jmn}
   Take a complete orthonormal system $\set{\mathfrak{e}_k}$ in $\mathfrak{H}_{\Cnum}$. For two sequences $\mathbf{m}=\set{m_k}_{k=1}^{\infty},\,\mathbf{n}=\set{n_k}_{k=1}^{\infty}$ of nonnegative integrals with finite sum, define a complex Fourier-Hermite polynomial
  \begin{equation}\label{fourier}
    \mathbf{J}_{\mathbf{m},\mathbf{n}}:=\prod_{k} \frac{1}{\sqrt{2^{m_k+n_k}m_k!n_k!}}J_{m_k,n_k}(\sqrt2 Z(\mathfrak{e}_k)).
  \end{equation}
\end{dfn}
\begin{rem} Let $\xi_k=Z(\mathfrak{e}_k)$ and we use the notation $\xi^{\mathbf{m}}\bar{\xi}^{\mathbf{n}}=\prod_k\xi_k^{m_k}\bar{\xi}_k^{n_k}$, then the right hand side of (\ref{fourier}) is exactly the Wick product $:\xi^{\mathbf{m}}\bar{\xi}^{\mathbf{n}}:$ ( please refer to Example 3.31 and Example 3.32 of \cite[p31]{jan}).

\end{rem}

In the following key proposition, the basis of $(m,n)$-th Wiener-It\^{o} chaos $\mathscr{H}_{m,n}(Z)$ and an isometry mapping from $\FH_{\Cnum}^{\odot m}\otimes \FH_{\Cnum}^{\odot n}$ onto $\mathscr{H}_{m,n}$ are given.
\begin{prop}\label{ppp1}
  Let $\mathbf{m},\,\mathbf{n}$ and $\mathfrak{e}_k$ be as in Definition~\ref{jmn}. Then:
      \begin{itemize}
    \item[\textup{(i)}]  For any $m,n\ge0$ the random variables
   \begin{equation}\label{eq4}
      \set{\mathbf{J}_{\mathbf{m},\mathbf{n}}:\abs{\mathbf{m}} =m,\,\abs{\mathbf{n}}=n }
   \end{equation}
   form a complete orthonormal system in $\mathscr{H}_{m,n}(Z)$.
      \item[\textup{(ii)}] The linear mapping
\begin{equation}\label{mulint}
   \mathscr{I}_{m,n}(\mathrm{symm}(\otimes_{k=1}^{\infty}\mathfrak{e}_k^{\otimes m_k})\otimes \mathrm{symm}(\otimes_{k=1}^{\infty}\bar{\mathfrak{e}}_k^{\otimes n_k}))=\sqrt{\mathbf{m}!\mathbf{n}!}\mathbf{J}_{\mathbf{m},\mathbf{n}}
\end{equation}
provides an isometry from the tensor product $\FH_{\Cnum}^{\odot m}\otimes \FH_{\Cnum}^{\odot n}$, equipped with the norm $\sqrt{m!n!}\norm{\cdot}_{\FH_{\Cnum}^{\otimes (m+n)}}$, onto the $(m,n)$-th Wiener-It\^{o} chaos $\mathscr{H}_{m,n}(Z)$.
  \end{itemize}
\end{prop}

Proof of Proposition~\ref{ppp1} is presented in Section~\ref{sec4}.
\begin{dfn}\label{imn}
For any $f\in \FH_{\Cnum}^{\odot m}\otimes \FH_{\Cnum}^{\odot n}$, we call $\mathscr{I}_{m,n}(f)$ the complex multiple Wiener-It\^{o} integral of $f$ with respect to $Z$ .
\end{dfn}

\begin{rem}
   Janson gives another complex multiple integrals $I_n$ from the viewpoint of Gaussian Hilbert space formally \cite[Theorem 7.52]{jan}. Actually, using the above Definition~\ref{jmn} and Proposition~\ref{ppp1}, his definition is only equivalent to a linear isometric mapping from $\bigoplus_{p+q=n}\FH_{\Cnum}^{\odot p}\otimes \FH_{\Cnum}^{\odot q}$ onto $\bigoplus_{p+q=n}\mathscr{H}_{p,q}(Z)$.
    In our opinion, (\ref{mulint}) matches the theory of It\^{o}'s complex multiple integrals \cite{ito} better (see (\ref{itomulitp}) and Remark~\ref{rm9} for the reason to call ``integrals").
\end{rem}


\begin{thm}{\bf (Complex Wiener-It\^{o} chaos decomposition)}\label{chaos}
\begin{itemize}
    \item[\textup{(i)}]The linear space generated by the class
    \begin{equation}\label{jmn2}
       \set{J_{m,n}(Z(\mathfrak{h})):\, m,n\ge 0, \mathfrak{h}\in \mathfrak{H}_{\Cnum},\norm{\mathfrak{h}}_{\mathfrak{H}_{\Cnum}}=\sqrt2}
    \end{equation}
   is dense in $L_{\Cnum}^q(\Omega,\sigma(X,Y),P)$ for every $q\in [1,\infty)$.
    \item[\textup{(ii)}] One has that $ L_{\Cnum}^2(\Omega,\sigma(X,Y),P)=\bigoplus^{\infty}_{m=0}\bigoplus^{\infty}_{n=0}\mathscr{H}_{m,n}.$
This means that every random variable $F\in  L_{\Cnum}^2(\Omega,\sigma(X,Y),P)$ admits a unique expansion of the type $F=\sum_{m=0}^\infty\sum_{n=0}^\infty F_{m,n}$, where $F_{m,n}\in \mathscr{H}_{m,n}, F_{0,0}=E[F]$ and the series converges in $ L_{\Cnum}^2(\Omega)$.
 \end{itemize}
\end{thm}
One can give several different proofs of the above theorem along the line of \cite{np,Nua} or \cite{jan}. In Subsection~\ref{sec4_1}, we will give a simple proof using the connection between the real Wiener-It\^{o} chaos and the complex Wiener-It\^{o} chaos based on Theorem~\ref{co2}.

\section{Relation between real and complex Wiener-It\^{o} chaos}\label{sec3}

\subsection{A product formula and relation between real multiple integrals and complex multiple integrals}\label{sec23}

To establish the connection between real multiple integrals and complex multiple integrals, we need a third real  isonormal Gaussian process $W$ over  the Hilbert space direct sum of the spaces $\FH$ and $\FH$.

Suppose $h,f\in \FH$, then denote $(h,f)$ the Cartesian product (or say, the order pair) of $\FH$ and $\FH$. Denote by $\FH\oplus \FH$ the Hilbert space direct sum of the spaces $\FH$ and $\FH$ with the natural inner product (see \cite[p48]{rs}), i.e., for any $h_1,h_2,f_1,f_2\in\FH$,
   \begin{align*}
      \innp{(h_1,f_1),\,(h_2,f_2)}_{\FH\oplus \FH}=\innp{h_1,h_2}_{\FH}+\innp{f_1,\,f_2}_{\FH}.
   \end{align*}
With respect to this inner product, $\FH\oplus \FH$ is a Hilbert space. We write $W=\set{W(h,f):\,h,f\in \FH}$ the isonormal Gaussian process over $\FH\oplus \FH$ and denote by $\mathcal{H}_n(W)$ the $n$-th Wiener-It\^{o} chaos of $W$.

For any $0<\theta_{n}<\dots<\theta_0<\pi$, denote a $(n+1)\times (n+1)$ matrix
\begin{align}\label{matr}
   \tensor{M}&=\tensor{M}(\theta_{0},\,\dots,\,\theta_n)\nonumber \\
     &=\left[
\begin{array}{lllll}
     (\sin\theta_0)^n & \,\,{n \choose 1}(\sin{\theta_0})^{n-1}\cos\theta_0   & \dots  &\,\, {n \choose n-1}\sin{\theta_0}(\cos\theta_0)^{n-1}&\,\, (\cos\theta_0)^{n}\\
     (\sin\theta_1)^n &\,\, {n \choose 1}(\sin{\theta_1})^{n-1}\cos\theta_1   & \dots  &\,\, {n \choose n-1}\sin{\theta_1}(\cos\theta_1)^{n-1}&\,\, (\cos\theta_1)^{n}\\
         \hdotsfor{5}\\
     (\sin\theta_n)^n &\,\, {n \choose 1}(\sin{\theta_n})^{n-1}\cos\theta_n   & \dots  &\,\, {n \choose n-1}\sin{\theta_n}(\cos\theta_n)^{n-1}&\,\, (\cos\theta_n)^{n}
\end{array}
\right]
\end{align}

\begin{thm}\label{mth00}{\bf (A product formula of real multiple integrals)}
Let Definition~\ref{nott01},\,\ref{nott02},\,\ref{imn} prevail. Denote by $  \mathcal{H}_{k}(X)\mathcal{H}_{l}(Y)$ the closed linear subspace of $L^2(\Omega, \mathcal{F},P)$ generated by the random variables of the type
\begin{equation*}
   \set{H_{k}(X(f))H_{l}(Y(g)):\,\norm{f}_{\FH}=\norm{g}_{\FH}=1 }.
\end{equation*}
Suppose that $\norm{f}_{\FH}^2+\norm{g}_{\FH}^2=1$, then
    \begin{equation}\label{man000}
   H_n(X(f) +Y(g))=\sum_{l=0}^n {n\choose l} \norm{f}^l\norm{g}^{n-l} H_l\big(\frac{X(f)}{\norm{f}}\big)H_{n-l}\big(\frac{Y(g)}{\norm{g}}\big).
\end{equation}
Suppose that $\norm{f}_{\FH}=\norm{g}_{\FH}=1$, for any fixed $0<\theta_{n}<\dots<\theta_0<\pi$, then
    \begin{equation}\label{man001}
  {H}_{l}(X(f)){H}_{n-l}(Y(g))=\sum_{k} \tensor{M}^{-1}_{l,k} H_n\big(\cos\theta_k X(f)+\sin\theta_k Y(g)\big),
  \end{equation} where $\tensor{M}^{-1}_{l,k}$ is the $(l,k)$-entry of $\tensor{M}^{-1}$, the inverse of $\tensor{M}$ (see (\ref{matr})).
That is to say, $\mathcal{H}_n(W)$, the $n$-th Wiener-It\^{o} chaos of $W$ satisfies
    \begin{equation}\label{typ1}
   \mathcal{H}_{n}(W)=\bigoplus_{k+l=n}\mathcal{H}_{k}(X)\mathcal{H}_{l}(Y).
\end{equation}
\end{thm}
\begin{thm}\label{co2}{\bf (The connection between real and complex Wiener-It\^{o} chaos)}\\
Let Definition~\ref{nott01},\ref{nott02},\ref{imn} prevail.
Suppose that $\norm{f}_{\FH}^2+\norm{g}_{\FH}^2=1$, $\norm{\tilde{f}}_{\FH}^2+\norm{\tilde{g}}_{\FH}^2=1$, then  for any fixed $\theta\in \Rnum$,
    \begin{align}\label{man010}
   H_n\big(X(f) +Y(g)\big)+\mi  H_n\big(X(\tilde{f}) +Y(\tilde{g})\big)\nonumber \\
   = \sum_{k=0}^n \, d_k \big(J_{k,n-k}(Z(\mathfrak{h}))+\mi J_{k,n-k}(Z(\tilde{\mathfrak{h}}))\big),
\end{align} where 
$\mathfrak{h}=\sqrt2 e^{\mi \theta}(f-\mi g),\,\tilde{\mathfrak{h}}=\sqrt2 e^{\mi \theta}(\tilde{f}-\mi \tilde{g})$,
and
\begin{align}\label{dk}
   d_k=\frac{1}{2^n}\sum_{r+s=k}(-1)^s \sum_{l=0}^n {n \choose l}{l \choose r}{n-l\choose s}(\cos\theta)^l(\mi \cdot\sin\theta)^{n-l}.
\end{align}

Suppose that $\FH_{\Cnum}\ni\mathfrak{h} $ with $\norm{\mathfrak{h}}_{\FH_{\Cnum}}=\sqrt{2}$, then
\begin{equation}\label{man011}
  J_{k,n-k}(Z(\mathfrak{h}))= \sum_{i=0}^n\tilde{c}_i H_n(X(f_i)+Y(g_i)),
\end{equation}
where 
$f_i+\mi g_i=\frac{1}{\sqrt2} e^{\mi \theta_i}\bar{\mathfrak{h}}$,
and
\begin{equation}\label{ci}
   \tilde{c}_i=\sum_{j=0}^n\tensor{M}^{-1}_{j,i}{\mi^{n-k}}\sum_{r+s=j}{k \choose r}{n-k \choose s}(-1)^{n-k-s}.
\end{equation}

That is to say, the complexification of $\mathcal{H}_n(W)$ satisfies
   \begin{equation}\label{eq3}
   \mathcal{H}^{\Cnum}_n(W):=\mathcal{H}_n(W)+\mi \mathcal{H}_n(W)=\bigoplus_{k+l=n}\mathscr{H}_{k,l}(Z).
\end{equation}
\end{thm}
 Proofs of Theorem~\ref{mth00}-\ref{co2} are presented in Section~\ref{sec4}.
 \begin{rem}
 Substituting (\ref{typ1}) into (\ref{eq3}), we get
 \begin{equation*}
\bigoplus_{k+l=n}\mathscr{H}_{k,l}(Z) = \bigoplus_{i+j=n}\big[\mathcal{H}_{i}(X)\mathcal{H}_{j}(Y)+\mi \mathcal{H}_{i}(X)\mathcal{H}_{j}(Y)\big],
 \end{equation*}which can be regarded as the infinite dimensional version of Equality (\ref{h2j}).
 \end{rem}

\begin{thm}\label{pp2} 
Suppose $\varphi\in \FH_{\Cnum}^{\odot m}\otimes \FH_{\Cnum}^{\odot n}$ and $F=\mathscr{I}_{m,n}(\varphi)=U+\mi V$. Then
     there exist real $u,\,v\in (\FH\oplus \FH)^{\odot (m+n)}$ such that
      \begin{align}
        U&=\mathcal{I}_{m+n}(u),\quad V=\mathcal{I}_{m+n}(v),\label{eeq}
      \end{align}
      where $\mathcal{I}_p(g)$ is the $p$-th real Wiener-It\^{o} multiple integral of $g$ with respect to $W$. And if $m\neq n$ then
       \begin{align}
         E[U^2]&=E[V^2], \quad E[UV]=(m+n)!\innp{u,\,v}_{(\FH\oplus\FH)^{\otimes (m+n)}}=0.\label{inn}
      \end{align}
\end{thm}
This theorem can be seen as the infinite dimensional version of Proposition~\ref{2dims}, the proof is presented in Section~\ref{sec4}.

\subsection{Ito's complex multiple integrals revisited}\label{exp2}
If $\FH$ is a real separable Hilbert space $L^2(T,\mathcal{B},\mu)$ and $\mu$ is a non-atomic measure, then the definition of $\mathscr{I}_{m,n}$ (see Definition~\ref{imn} ) coincides with a multiple Wiener-It\^{o} integrals defined by It\^{o} \cite{ito}. In fact, at first, It\^{o} gave a continuous complex normal random measure $\mathbf{M}=\set{M(B):B\in \mathcal{B},\mu(B)<\infty}$ on $(T,\mathcal{B})$, such that, for every $B, C\in \mathcal{B}$ with finite measure,
   \begin{equation*}
      E[M(B)\overline{M(C)}]=\mu(B\cap C).
   \end{equation*}
 Next, for the off-diagonal simple function $f\in {\FH}_{\Cnum}^{\otimes m}\otimes {\FH}_{\Cnum}^{\otimes n}$ of the form
\begin{equation}\label{sfc}
   f(t_1,\dots,t_m,s_1,\dots,s_n)=\sum a_{i_1\dots i_m j_1\dots j_n} \mathbf{1}_{E_{i_1}\times\dots\times E_{i_m}\times E_{j_1}\times\dots\times E_{j_n} },
\end{equation}
with $\mathbf{1}_{B}(\cdot) $ the characteristic function of the set $B$, he defined the multiple integrals $I_{m,n}(f)$ by
\begin{equation}
   I_{m,n}(f)=\sum a_{i_1\dots i_m j_1\dots j_n} M(E_{i_1})\dots M(E_{i_m})\overline{M(E_{j_1})}\dots \overline{ M(E_{j_n})}.
\end{equation}
And then by density argument, he extended the multiple integrals to any $f\in {\FH}_{\Cnum}^{\otimes m}\otimes {\FH}_{\Cnum}^{\otimes n}$,
\begin{equation*}
 I_{m,n}(f)=\int\cdots\int f(t_1,\dots,t_m,s_1\dots,s_n) \dif M(t_{1})\dots \dif M(t_{m})\overline{\dif M(s_{1})}\dots \overline{ \dif M(s_{n})}.
\end{equation*}
Moreover, It\^{o} established the relation between complex multiple integrals and complex Hermite polynomials: suppose that $ \mathfrak{h}_1(t),\dots,\mathfrak{h}_l(t)$ be any orthonormal system in $\FH_{\Cnum}$ and $\alpha_i,\,\beta_j=1,\dots,l$, then
\begin{align}
  &\int\cdots\int \mathfrak{h}_{\alpha_1}(t_1)\cdots \mathfrak{h}_{\alpha_m}(t_m)\overline{\mathfrak{h}_{\beta_1}(s_1)}\dots\overline{\mathfrak{h}_{\beta_n}(s_n)} \dif M(t_{1})\dots \dif M(t_{m})\overline{\dif M(s_{1})}\dots \overline{ \dif M(s_{n})}\nonumber \\
  &=\prod_{k=1}^l \, {2^{-\frac{m_k+n_k}{2}}}J_{m_k,n_k}(\sqrt2 z_k),\label{itomulitp}
\end{align}
where $z_k=\int \mathfrak{h}_k(t) \dif {M}(t),\,k=1,\dots,l$ and $m_k,\,n_k$ are the number of $k$ appearing in $\alpha_i$ and $\beta_j$ respectively.
\begin{rem}\label{rm9}
Here the notation $2^{-\frac{i+j}{2}}J_{i,j}(\sqrt2 z)$ is exactly the notation $H_{i,j}(z,\bar{z})$ of \cite{ito} by It\^{o}. It follows from (\ref{mulint}) and (\ref{itomulitp}) that $\mathscr{I}_{m,n}$ coincides with $I_{m,n}$.
\end{rem}
Now we turn to express Eq.(\ref{man000})-(\ref{man001}) and Eq.(\ref{man010})-(\ref{man011}) in terms of It\^{o}'s theory. Set $\mathbf{M}=\frac{1}{\sqrt2}[\mathbf{M}_1+\mi \mathbf{M}_2]$. Then $\mathbf{M}_1,\, \mathbf{M}_2$ are two real independent continuous normal system such that , for every $B, C\in \mathcal{B}$ with finite measure, $E[M_1(B)M_1(C)]=E[M_2(B)M_2(C)]=\mu(B\cap C)$.  Set $\widehat{T}=\set{1,2}\times T,\,\mathcal{B}(\widehat{T})=\mathcal{B}(\set{1,2}\times {T})$. And set
\begin{equation*}
  \widehat{{M}}(B)={M}_1(B_1)+ {M}_2(B_2),\quad \forall B=\big(\set{1}\times B_1 \big) \bigcup \big(\set{2}\times B_2 \big)\in \mathcal{B}(\widehat{T}).
\end{equation*}
Then $\widehat{\mathbf{M}}=\set{\widehat{{M}}(B):\,B=\big(\set{1}\times B_1 \big) \bigcup \big(\set{2}\times B_2 \big),\, \mu(B_1)+\mu(B_2)<\infty} $ is a normal random measure on $(\hat{T},\,\mathcal{B}(\widehat{T}))$ and
$ L^2(\widehat{T}) =L^2(T)\oplus L^2(T).$

For any $\widehat{f}=(f_1,f_2)$ with $f_i\in \FH,\,i=1,2$.
Suppose $\norm{f_1}_{\FH}^2+\norm{f_2}_{\FH}^2=1$, then we have that
\begin{equation*}
  \int \widehat{f} \dif \widehat{{M}} =\int f_1\dif {M}_1 +\int f_2\dif {M}_2,
\end{equation*}
and Eq.(\ref{man000}) means that
\begin{align*}
 &\int\cdots \int  \widehat{f}^{\otimes n} \dif \widehat{{M}}(t_1)\cdots\dif \widehat{{M}}(t_n)\\
 &= H_n (\int \widehat{f} \dif \widehat{{M}} ) \\
&=\sum_{l=0}^n {n\choose l} \norm{f_1}^l\norm{f_2}^{n-l} H_l\big(\frac{\int f_1\dif {M}_1}{\norm{f_1}}\big)H_{n-l}\big(\frac{\int f_2\dif {M}_2}{\norm{f_2}}\big)\\
&=\sum_{l=0}^n {n\choose l} \int\cdots \int {f_1}^{\otimes l}\dif {{M}_1}(t_1)\cdots\dif{{M}_1}( t_l) \int\cdots \int  {f_2}^{\otimes (n-l)}\dif {{M}_2}( t_{l+1})\cdots\dif {{M}_2}(t_n).
\end{align*}
Suppose that $\norm{f_1}_{\FH}^2=\norm{f_2}_{\FH}^2=1$, then Eq.(\ref{man001}) means that
\begin{align*}
&\int\cdots \int {f_1}^{\otimes l}\dif {{M}_1}(t_1)\cdots\dif {{M}_1}(t_l) \int\cdots \int  {f_2}^{\otimes (n-l)}\dif {{M}_2}(t_{l+1})\cdots\dif {{M}_2}(t_n)\\
&=   {H}_{l}\big (\int f_1\dif {M}_1\big){H}_{n-l}\big(\int f_2\dif {M}_2 \big)\\
&=\sum_{k} \tensor{M}^{-1}_{l,k} H_n\big(\cos\theta_k \int f_1\dif {M}_1+\sin\theta_k \int f_2\dif {M}_2\big)\\
&=\sum_{k} \tensor{M}^{-1}_{l,k}  \int  \widehat{f^{(k)}}^{\otimes n} \dif \widehat{{M}}(t_1)\cdots\dif \widehat{{M}}(t_n),
\end{align*} where $ \widehat{f^{(k)}}= (f^{(k)}_1,\,f^{(k)}_2)=(\cos\theta_k f_1,\, \sin\theta_k f_2)  $.

Moreover, let $\widehat{g}=(g_1,\,g_2)$, and suppose that $\norm{f_1}_{\FH}^2+\norm{f_2}_{\FH}^2=1,\,\norm{g_1}_{\FH}^2+\norm{g_2}_{\FH}^2=1$, then Eq.(\ref{man010}) means that
\begin{align*}
     &\int\cdots \int  (\widehat{f}^{\otimes n} +\mi  \widehat{g}^{\otimes n}) \,\dif \widehat{{M}}(t_1)\cdots\dif \widehat{{M}}(t_n)\\
 &= H_n (\int \widehat{f} \dif \widehat{{M}} ) +\mi H_n (\int \widehat{g} \dif \widehat{{M}} )\\
 &= \sum_{k=0}^n \, d_k \big(J_{k,n-k}(\int \mathfrak{h} \dif {M})+\mi J_{k,n-k}(\int \tilde{\mathfrak{h}} \dif {M})\big)\\
 &=\sum_{k=0}^n \, d_k \int\cdots\int  \big(\mathfrak{h}^{\otimes k}\otimes  \overline{\mathfrak{h}}^{\otimes (n-k)}+\mi\, \tilde{\mathfrak{h}}^{\otimes k}\otimes  \overline{\tilde{\mathfrak{h}}}^{\otimes (n-k)}\big)\dif M(t_{1})\cdots \dif M(t_{m})\overline{\dif M(s_{1})}\cdots \overline{ \dif M(s_{n})} ,
\end{align*}
where $\mathfrak{h}=\sqrt2 e^{\mi \theta}(f_1-\mi f_2),\,\tilde{\mathfrak{h}}=\sqrt2 e^{\mi \theta}(g_1-\mi g_2)$.

Suppose that $\mathfrak{h}\in \FH_{\Cnum}$ with $\norm{\mathfrak{h}}_{\FH_{\Cnum}}=\sqrt2$,  then Eq.(\ref{man011}) means that
\begin{align*}
&\int\cdots\int \mathfrak{h}^{\otimes k}\otimes  \overline{\mathfrak{h}}^{\otimes (n-k)} \dif M(t_{1})\dots \dif M(t_{m})\overline{\dif M(s_{1})}\cdots \overline{ \dif M(s_{n})}\\
&=  J_{k,n-k}(\int \mathfrak{h} \dif {M})= \sum_{i=0}^n\tilde{c}_i H_n(\int \widehat{f}_i \dif \widehat{{M}})\\
&= \sum_{i=0}^n\tilde{c}_i \int\cdots \int \widehat{f^{(i)}}^{\otimes n} \dif \widehat{{M}}(t_1)\cdots\dif \widehat{{M}}(t_n),
\end{align*} where  $\widehat{f^{(i)}}=(f_i,g_i)$ and $f_i+\mi g_i=\frac{1}{\sqrt2} e^{\mi \theta_i}\bar{\mathfrak{h}}$.

\begin{exmp}
Let $(B_1(t),B_2(t))$ denote 2-dimensional Brownian motion on $t\in[0,\,\infty)$. We put $\zeta_t:=\frac{B_1(t)+\mi B_2(t)}{\sqrt2}$. $\zeta_t$ is called complex Brownian motion. Extending Theorem 9.6.9 in the textbook by Kuo \cite{guo} to 2-dimensional Brownian motion, we have
   \begin{align*}
      H_n(\int \widehat{f} \dif \widehat{{M}} )&=H_n\big(\frac{1}{\sqrt2}\int_0^{\infty} f_1(t)\dif B_1(t) +\frac{1}{\sqrt2}\int_0^{\infty} f_2(t)\dif B_2(t) \big)\\
      &=\frac{n!}{2^{n/2}}\sum_{i_1,\dots,i_n=1}^2 \int_0^{\infty}\int_0^{t_n}\cdots \int_0^{t_2} f_{i_1}(t_1)\cdots f_{i_n}(t_n)\dif B_{i_1}(t_1)\cdots\dif B_{i_n}(t_n).
   \end{align*}
Thus, we can express Eq.(\ref{man000}) in It\^{o}'s {\it iterated} integrals as follows
    \begin{align*}
       & \sum_{i_1,\dots,i_n=1}^2 \int_0^{\infty}\int_0^{t_n}\cdots \int_0^{t_2} f_{i_1}(t_1)\cdots f_{i_n}(t_n)\dif B_{i_1}(t_1)\cdots\dif B_{i_n}(t_n)\\
      &= \sum_{l=0}^n  \int_0^{\infty}\int_0^{t_l}\cdots \int_0^{t_2} {f_1}(t_1)\cdots{f_1}(t_l)\dif B_1(t_1)\cdots\dif B_1(t_l)\\
      &\quad \times \int_0^{\infty}\int_0^{t_{n-l}}\cdots \int_0^{t_{l+2}}  {f_2}(t_{l+1})\cdots{f_2}(t_n)\dif B_2(t_{l+1})\cdots \dif B_2(t_{n}).
    \end{align*}
We express Eq.(\ref{man001}) in It\^{o}'s {\it iterated} integrals as follows
\begin{align*}
& \int_0^{\infty}\int_0^{t_l}\cdots \int_0^{t_2} {f_1}(t_1)\cdots{f_1}(t_l)\dif B_1(t_1)\cdots\dif B_1(t_l)\\
      &\quad \times \int_0^{\infty}\int_0^{t_{n-l}}\cdots \int_0^{t_{l+2}}  {f_2}(t_{l+1})\cdots{f_2}(t_n)\dif B_2(t_{l+1})\cdots \dif B_2(t_{n})\\
&={n \choose l}\sum_{k} \tensor{M}^{-1}_{l,k} \sum_{i_1,\dots,i_n=1}^2 \int_0^{\infty}\int_0^{t_n}\cdots \int_0^{t_2} f^{(k)}_{i_1}(t_1)\cdots f^{(k)}_{i_n}(t_n)\dif B_{i_1}(t_1)\cdots\dif B_{i_n}(t_n),
\end{align*}
We express Eq.(\ref{man010}) in It\^{o}'s iterated integrals as follows
\begin{align*}
& \sum_{i_1,\dots,i_n=1}^2 \int_0^{\infty}\int_0^{t_n}\cdots \int_0^{t_2}[ f_{i_1}(t_1)\cdots f_{i_n}(t_n)+\mi g_{i_1}(t_1)\cdots g_{i_n}(t_n)]\dif B_{i_1}(t_1)\cdots\dif B_{i_n}(t_n)\\
 &=\frac{2^{n/2}}{n!}\sum_{k=0}^n \, d_k \int\cdots\int  \big(\mathfrak{h}^{\otimes k}\otimes  \overline{\mathfrak{h}}^{\otimes (n-k)}+\mi\, \tilde{\mathfrak{h}}^{\otimes k}\otimes  \overline{\tilde{\mathfrak{h}}}^{\otimes (n-k)}\big)\dif \zeta(t_{1})\cdots \dif \zeta(t_{m})\overline{\dif \zeta(s_{1})}\cdots \overline{ \dif \zeta(s_{n})} ,
\end{align*} where $\norm{f_1}_{\FH}^2+\norm{f_2}_{\FH}^2=1,\,\norm{g_1}_{\FH}^2+\norm{g_2}_{\FH}^2=1$ and $\mathfrak{h}=\sqrt2 e^{\mi \theta}(f_1-\mi f_2),\,\tilde{\mathfrak{h}}=\sqrt2 e^{\mi \theta}(g_1-\mi g_2)$.

We express Eq.(\ref{man011}) in It\^{o}'s iterated integrals as follows
\begin{align*}
    &  \int\cdots\int \mathfrak{h}^{\otimes k}\otimes  \overline{\mathfrak{h}}^{\otimes (n-k)} \,\dif \zeta(t_{1})\cdots \dif \zeta(t_{m})\overline{\dif \zeta(s_{1})}\cdots \overline{ \dif \zeta(s_{n})} \\
    &= \frac{n!}{2^{n/2}} \sum_{j=0}^n\tilde{c}_j\sum_{i_1,\dots,i_n=1}^2 \int_0^{\infty}\int_0^{t_n}\cdots \int_0^{t_2} f^{(j)}_{i_1}(t_1)\cdots f^{(j)}_{i_n}(t_n)\dif B_{i_1}(t_1)\cdots\dif B_{i_n}(t_n),
\end{align*} where  $\widehat{f}^{(j)}=({f}^{(j)}_1,{f}^{(j)}_2)$ and ${f}^{(j)}_1+\mi {f}^{(j)}_2=\frac{1}{\sqrt2} e^{\mi \theta_j}\bar{\mathfrak{h}}$.
\end{exmp}
\begin{exmp}
  Let $\set{\zeta_j=\xi^{(1)}_j+\mi \xi^{(2)}_j:\,j\ge 1}$ be a sequence of i.i.d. symmetric complex normal random variables with variance $1$. Set $T=\set{1,2\dots,n,\dots}$, $\sharp$ be the counting measure on $T$ and $\mathbf{M}(E)=\sum\limits_{j\in E\subset T} \zeta_j$, which is the complex normal random measure on $(T,\sharp)$.

  This system is not included in It\^{o}'s framework, because $(T,\sharp)$ is not continuous (or say non-atomic, see \cite{ito} and \cite{Nua} ). Specifically,
  set $\FH=l^2$,  $f_1,\,f_2\in \FH$ and $\widehat{f}=f_1+\mi f_2$ with $\norm{f_1}^2_{\FH}+\norm{f_2}^2_{\FH}=1$, by direct computation,
  \begin{align*}
   & \sum_{t_1\dots t_n,s_1\dots s_m}\widehat{f}(t_1)\cdots\widehat{f}(t_n)\overline{\widehat{f}(s_1)}\cdots\overline{\widehat{f}(s_m)}\zeta_{t_1}\cdots\zeta_{t_n}
   \overline{\zeta_{s_1}}\cdots\overline{\zeta_{s_m}}\\
   \neq& H_{n,m}\big(\sum_{n}\widehat{f}(n)\zeta_n, \sum_{n}\overline{\widehat{f}(n)}\overline{\zeta_n}\big)\\
   =& 2^{-\frac{n+m}{2}}J_{n,m}\big(\sqrt{2}\sum_{n}\widehat{f}(n)\zeta_n\big).\qquad \text{(see Remark~\ref{rm9})}
  \end{align*}
  This means that It\^{o}'s multiple integrals with discrete time do not coincide with the chaos decomposition with respect to real or complex Hermite polynomials (see Definition~\ref{df27} ). For an alternative theory of stochastic integrals with discrete time and the corresponding chaos decomposition, please refer to the monograph by Privault \cite{hpp}.


\end{exmp}
\section{Proof of theorems}\label{sec4}
\subsection{Proof of Proposition~\ref{ppp1} and Theorem~\ref{mth00}-\ref{pp2}}\label{sec4_1}

\noindent{\it Proof of Proposition~\ref{ppp1}.\,}
(i): We follow the arguments of the real Wiener-It\^{o} chaos \cite[Page 7]{Nua}. Concretely, denote $\mathcal{P}_{m,n}$ \cite[Page 6]{Nua} be the closure of the linear space generated by the class
\begin{align}\label{ply}
   \set{p(Z(\mathfrak{h}_1),\dots, Z(\mathfrak{h}_j)):\,j\ge 1,\mathfrak{h}_1,\dots,\mathfrak{h}_j\in \FH_{\Cnum}},
\end{align}
where $p$ is a polynomial in complex variables $z_1,\dots, z_j$ of degree less than $m$ and $\bar{z}_1,\dots, \bar{z}_j$ of degree less than $n$ (for short, say the degree of $p$ less than or equal to $(m,n)$).
We claim that
\begin{align}\label{pmn}
   \mathcal{P}_{m,n}=\bigoplus^{m}_{k=0}\bigoplus^{n}_{l=0}\mathscr{H}_{k,l}(Z).
\end{align}
Indeed, since $J_{k,l}(z)$ is a polynomial in complex variable $z$ of degree $k$ and $\bar{z}$ of degree $l$ (see \cite[Theorem 2.15]{cl} or
\cite[Example 3.31 ]{jan}), the inclusion $\bigoplus^{m}_{k=0}\bigoplus^{n}_{l=0}\mathscr{H}_{k,l}(Z)\subset \mathcal{P}_{m,n}$ is immediate. To prove the converse inclusion, it is enough to check $ \mathcal{P}_{m,n}$ is orthogonal to all $\mathscr{H}_{k,l}(Z)$ for $k> m$ or $l>n$, i.e., to show that for any $\mathfrak{h}\in\FH_{\Cnum}$ $E[p(Z(\mathfrak{h}_1),\dots, Z(\mathfrak{h}_j))J_{l,k}(Z(\mathfrak{h}))]=0 $ . By the Gram-Schmidt orthogonalization process, we can suppose that $\mathfrak{h}\in\set{\mathfrak{h}_1,\dots,\mathfrak{h}_j}$ which is an orthonormal system. The equality \cite[Corollary 2.8]{cl}
\begin{align}\label{singlenomial}
  z^r\bar{z}^s&= \sum_{i=0}^{r\wedge s} {r \choose i}{s \choose i}i! 2^{i} J_{r-i,s-i}(z)
\end{align}
implies $E[Z(\mathfrak{h})^r\overline{Z(\mathfrak{h})}^s J_{l,k}(Z(\mathfrak{h})) ]=0$ when $r<l,\,s<k$. This ends the proof together with the independent property of $Z(\mathfrak{h}_i)$ for $i=1,\dots,j$.

Clearly, the random variables of the class $\set{\mathbf{J}_{\mathbf{m},\mathbf{n}}:\abs{\mathbf{m}} =m,\,\abs{\mathbf{n}}=n}$ belong to $ \mathcal{P}_{m,n}$ and are orthonormal system. Since $\set{\mathfrak{e}_i}$ is an orthonormal basis of $\FH_{\Cnum}$, every polynomial random variable $p(Z(\mathfrak{h}_1),\dots,Z(\mathfrak{h}_j))$ as (\ref{ply}) can be approximated by polynomials $q(Z(\mathfrak{e}_1),\dots,Z(\mathfrak{e}_r))$ with the degree of $q$ less than or equal to $(m,n)$. Thus Eq.(\ref{singlenomial}) implies that the random variables of the class $\set{\mathbf{J}_{\mathbf{m},\mathbf{n}}:\abs{\mathbf{m}} \le m,\,\abs{\mathbf{n}}\le n}$ are a basis of $\mathcal{P}_{m,n}$. But $\set{\mathbf{J}_{\mathbf{m},\mathbf{n}}:\abs{\mathbf{m}} = m,\,\abs{\mathbf{n}}= n}$ are orthogonal to $\mathcal{P}_{m-1,n}\bigcup\mathcal{P}_{m,n-1}$. Thus the class $\set{\mathbf{J}_{\mathbf{m},\mathbf{n}}:\abs{\mathbf{m}} = m,\,\abs{\mathbf{n}}= n}$ are a basis of $\mathscr{H}_{m,n}(Z)$.

(ii): The isometry property is deduced from $\norm{\mathrm{symm}(\otimes_{k=1}^{\infty}\mathfrak{e}_k^{\otimes m_i})}^2_{\FH_{\Cnum}^{\otimes m}}=\frac{\mathbf{m}!}{m!}$ \cite[Page 8]{Nua}. Since by (i) the span of $\set{\mathbf{J}_{\mathbf{m},\mathbf{n}}:\abs{\mathbf{m}} =m,\,\abs{\mathbf{n}}=n}$ generates $\mathscr{H}_{m,n}(Z)$ and since linear combinations of vectors of the
type $\mathrm{symm}(\otimes_{k=1}^{\infty}\mathfrak{e}_k^{\otimes m_k})\otimes \mathrm{symm}(\otimes_{k=1}^{\infty}\bar{\mathfrak{e}}_k^{\otimes n_k}) $ are dense in $\FH_{\Cnum}^{\odot m}\otimes \FH_{\Cnum}^{\odot n}$, we have that the mapping between $\FH_{\Cnum}^{\odot m}\otimes \FH_{\Cnum}^{\odot n} $ and $\mathscr{H}_{m,n}(Z)$ is onto.
{\hfill\large{$\Box$}}
\vskip 5mm

We cite a special determinant which is Problem 342 of \cite{pro}.
\begin{lem}\label{lm01} Let $\mathbb{F}[x,y]$ be the set of 2-variate polynomials over the field $\mathbb{F}$.
Set $f_i(a,b)\in \mathbb{F}[a,b]$ be the homogeneous polynomials of degree $i$ with $i=1,\dots,n$. Then the $n+1$ order determinant
   \begin{align*}
  &\quad   \left|
\begin{array}{lllll}
     f_n(a_0,b_0) & b_0f_{n-1}(a_0,b_0)   & \dots  & b_0^{n-1}f_1(a_0,b_0)& b_0^n\\
     f_n(a_1,b_1) & b_1f_{n-1}(a_1,b_1)   & \dots  & b_1^{n-1}f_1(a_1,b_1)& b_1^n\\
         \hdotsfor{5}\\
     f_n(a_{n},b_{n}) & b_{n}f_{n-1}(a_{n},b_{n})   & \dots  & b_{n}^{n-1}f_1(a_{n},b_{n})& b_{n}^n
\end{array}
\right|\\
&=c_{1}c_{2}\dots c_{n}\prod_{0\le i<j\le n}(a_i b_j-a_j b_i),
   \end{align*}
   where $c_i$ is the coefficient of $a^i$ in the polynomial $f_i(a,b)$.
\end{lem}
\vskip 0.5cm
\noindent{\it Proof of Theorem~\ref{mth00}.\,}
Eq.(\ref{man000}) can be directly induced from the invariant property of Hermite polynomials \cite{Ma,str}
\begin{equation}\label{hn0}
   H_n(x\cos \theta+y\sin\theta)=\sum_{l=0}^n {n\choose l} (\cos \theta)^l(\sin\theta)^{n-l} H_l(x)H_{n-l}(y).
\end{equation}
In fact, if $\norm{f}_{\FH}^2+\norm{g}_{\FH}^2=1$ we can choose $\theta\in\Rnum$ such that $\cos\theta=\norm{f}_{\FH},\,\sin\theta=\norm{g}_{\FH}$ and let $x=\frac{X(f)}{\norm{f}_{\FH}},\,y=\frac{Y(g)}{\norm{g}_{\FH}}$ in the above equation.

Now we turn to Eq.(\ref{man001}). Denote $a=\sin\theta,\, b=\cos\theta$ and let $f_l(a,b)={n \choose l}a^l$. If we choose $\theta=\theta_k,\,k=0,\dots,n$, then Eq.(\ref{hn0}) can be looked as a system of $n+1$ linear equations in $n+1$ unknowns $H_l(x)H_{n-l}(y),\,l=0,\dots,n$ and the matrix of coefficients of the system is
\begin{equation}\label{matr11}
   \tensor{M}=\big(b_k^{l}f_{n-l}(a_k,b_k)\big)_{k,l},
\end{equation} where $k,l=0,\dots,n,\,a_k=\sin\theta_k,\,b_k=\cos\theta_k$. It follows from Lemma~\ref{lm01} that the determinant of the coefficient matrix equals to
\begin{align*}
   \prod_{k=0}^n{n \choose k}\prod_{0\le i<j\le n}(a_i b_j-a_j b_i)=\prod_{k=0}^n{n \choose k}\prod_{0\le i<j\le n}\sin(\theta_i-\theta_j)\neq 0.
\end{align*}
That is to say, the coefficient matrix $\tensor{M}$ is invertible. Thus the unknowns $H_l(x)H_{n-l}(y)$ can be expressed as
\begin{equation*}
   H_l(x)H_{n-l}(y)=\sum_{k=0}^n \tensor{M}^{-1}_{l,k}H_n(x\cos\theta_k+y\sin\theta_k),\quad l=0,\dots,n.
\end{equation*}
Set $x=X(f),\,y=Y(g)$ in the above equation displayed, Eq.(\ref{man001}) is induced. Since $\norm{f\cos \theta_k}^2_{\FH}+\norm{g\sin \theta_k}^2_{\FH}=1$, Eq.(\ref{typ1}) can be deduced from the following
Proposition~\ref{p0401} and the definition of $\mathcal{H}_k(X)\mathcal{H}_l(Y)$.
{\hfill\large{$\Box$}}
\vskip 0.5cm

We divide the proof of Theorem~\ref{co2} into four propositions.
\begin{prop}\label{p0401} 
Suppose that $Y$ is an independent copy of the real isonormal Gaussian process $X$ over the Hilbert space $\FH$. Then a realization of the isonormal Gaussian process $W$ over the Hilbert space direct sum $\FH\oplus \FH$ is
\begin{align*}
      W(h,f)&=X(h)+Y(f),\quad\forall h,f\in\FH.
   \end{align*}
  That is to say, $\set{X(h)+Y(f):\,\forall h,f\in\FH}$ is a centered Gaussian family such that
 $\forall h_1,f_1,h_2,f_2\in \FH$,
     \begin{align*}
      E[(X(h_1)+Y(f_1))(X(h_2)+Y(f_2))]= \innp{(h_1,f_1),\,(h_2,f_2)}_{\FH\oplus \FH}.
   \end{align*}
\end{prop}
\begin{proof}
Note that $Y$ is an independent copy of the isonormal Gaussian process $X$, we have that $\sum_{j}a_j[X(h_j)+Y(f_j)]=X(\sum_{j}a_j h_j)+Y(\sum_{j}a_j f_j)$ is 1-dimensional centered normal random variable, where $a_j\in \Rnum$. Thus $\set{X(h)+Y(f):\,\forall h,f\in\FH}$ is a centered Gaussian family. The independent property of $X$ and $Y$ implies that
   \begin{align*}
     E[(X(h_1)+Y(f_1))(X(h_2)+Y(f_2))]&=\innp{h_1,\,h_2}_{\FH}+\innp{f_1,\,f_2}_{\FH}\\
     &= \innp{(h_1,f_1),\,(h_2,f_2)}_{\FH\oplus \FH}.
   \end{align*}
\end{proof}


\begin{prop}\label{pm10}
   If $\FH_{\Cnum}\ni \mathfrak{h}=u+\mi v$ such that $\norm{\mathfrak{h}}_{{\FH}_{\Cnum}}=\sqrt2$, $0<\theta_n<\cdots<\theta_0<\pi$ and $f_i+\mi g_i=\frac{1}{\sqrt2} e^{\mi \theta_i}\bar{\mathfrak{h}}$, then \begin{equation}\label{dp1}
   J_{k,n-k}(Z(\mathfrak{h}))=\sum_{i}\tilde{c}_i H_n(X(f_i)+Y(g_i))
\end{equation} holds, where $\tilde{c}_i$ is a constant depending on $\theta_k$ given by (\ref{ci}).
\end{prop}
\begin{proof}
(\ref{xh2}) and (\ref{zh}) imply that
\begin{align*}
   Z(\mathfrak{h})&=\frac{X_{\Cnum}(\mathfrak{h})+\mi Y_{\Cnum}(\mathfrak{h})}{\sqrt2}
   =\frac{1}{\sqrt2}[X(u)-Y(v)]+\frac{\mi }{\sqrt2} [X(v)+Y(u) ].
\end{align*}
 It follows from the fundamental relation Eq.(\ref{h2j}) and Theorem~\ref{mth00} that we have
\begin{align}
   J_{k,n-k}(Z(\mathfrak{h}))&= \sum_{j}c_j H_j(\RE(Z(\mathfrak{h})))H_{n-j}(\IM(Z(\mathfrak{h})))\nonumber\\
   &= \sum_{j}c_j \sum_{i} \tensor{M}^{-1}_{j,i} H_n\big(\cos\theta_i\RE(Z(\mathfrak{h}))+\sin\theta_i \IM(Z(\mathfrak{h}))\big), \nonumber\\
   &=\sum_{i}\tilde{c}_i H_n\big(\cos\theta_i\RE(Z(\mathfrak{h}))+\sin\theta_i \IM(Z(\mathfrak{h}))\big)   \label{jkl}
 \end{align} where $c_j={\mi^{n-k}}\sum_{r+s=j}{k \choose r}{n-k \choose s}(-1)^{n-k-s}$,\,$\tilde{c}_i=\sum_{j}c_j\tensor{M}^{-1}_{j,i}$ are constant.
Since
 \begin{align}
   \cos\theta_i \RE(Z(\mathfrak{h}))+\sin\theta_i \IM(Z(\mathfrak{h}))&= \frac{ \cos\theta_i}{\sqrt2}[X(u)-Y(v)]+ \frac{\sin\theta_i}{\sqrt2} [X(v)+Y(u) ]\nonumber \\
   &=X(f_i)+Y(g_i), \label{aaa}
 \end{align}
 and $\norm{f_i}_{\FH}^2+\norm{g_i}_{\FH}^2=\frac12\norm{\mathfrak{h}}^2_{\FH_{\Cnum}}=1$,
we have that
\begin{align*}
  H_n(\cos\theta_i\RE(Z(\mathfrak{h}))+ \sin\theta_i\IM(Z(\mathfrak{h})))=H_n(X(f_i)+Y(g_i))
\end{align*}
Inserting the above equation displayed into (\ref{jkl}), we get Eq.(\ref{dp1}).
\end{proof}
\begin{prop}\label{p0402}
   If $f,\,g\in \FH$ such that $\norm{f}^2_{\FH}+\norm{g}^2_{\FH}=1$ and $\mathfrak{h}=\sqrt2 e^{\mi \theta}(f-\mi g)$ with $\theta\in \Rnum $,
   then
    \begin{equation}\label{dp2}
  H_n(X(f)+Y(g))=\sum_{k=0}^n d_k J_{k,n-k}(Z(\mathfrak{h}))
\end{equation}
holds, where $d_k$ is a constant depending on $\theta$ given by (\ref{dk}).
\end{prop}
\begin{proof}
It follows from Eq.(\ref{aaa}), Eq.(\ref{hn0}) and the fundamental relation Eq.(\ref{h2j}) that
\begin{align*}
    H_n(X(f)+Y(g))&=H_n(\cos\theta\, \RE(Z(\mathfrak{h}))+\sin\theta\, \IM(Z(\mathfrak{h})))\\
    &=\sum_{l=0}^n {n \choose l}(\cos\theta)^l(\sin\theta)^{n-l}H_{l}(\RE(Z(\mathfrak{h}))H_{n-l}( \IM(Z(\mathfrak{h}))\\
    &=\sum_{k=0}^n \, d_k J_{k,n-k}(Z(\mathfrak{h})).
\end{align*}
\end{proof}
It follows from Proposition~\ref{p0401}-\ref{p0402} and the definition of $\mathcal{H}^{\Cnum}_n(W)$ that one has the following corollary.
\begin{cor}\label{ccc1} Let $\mathcal{H}^{\Cnum}_{n}(W)$ be as in Equality (\ref{hcx}). Then
   $\mathcal{H}^{\Cnum}_{n}(W)$ is also the closed linear subspace of $L_{\Cnum}^{2}(\Omega, \sigma(W) ,P)$ generated by the random variable of the type $\set{J_{k,l}(Z(\mathfrak{h})):k+l=n,\,\mathfrak{h}\in \FH_{\Cnum},\norm{\mathfrak{h}}_{\FH_{\Cnum}}=\sqrt2}$.
\end{cor}

\begin{prop}\label{pp3}
   Let $\zeta_1,\zeta_2\sim \mathcal{CN}(0,2)$ be jointly symmetric complex Gaussian (i.e., $c_1\zeta_1+c_2\zeta_2$ is a symmetric complex Gaussian variable for any $c_i\in \Cnum$ \cite{ito}.).\footnote{It is necessary that $E[\zeta_1\zeta_2]=0$.} Then for all $m_1,n_1,m_2,n_2\in \Nnum$:
   \begin{equation*}
      E[J_{m_1,n_1}(\zeta_1)\overline{J_{m_2,n_2}(\zeta_2)}]=\left\{
      \begin{array}{ll}
      m_1!n_1!(E[\zeta_1\bar{\zeta}_2])^{m_1}(E[\bar{\zeta}_1\zeta_2])^{n_1}, \quad &\text{if\quad} m_1=m_2,\,n_1=n_2\\
      0,\quad &\text{otherwise.}
      \end{array}
      \right.
   \end{equation*}
\end{prop}
\begin{proof}
By the Laplace transform of the jointly Gaussian distribution, we have that for all $s,t\in \Cnum$,
   \begin{align*}
      &E\big( \exp\set{s\bar{\zeta}_1+\bar{s}\zeta_1-2|s|^2}\exp\set{t\bar{\zeta}_2+\bar{t}\zeta_2-2|t|^2}\big)\\
      =&\exp\set{E(s\bar{\zeta}_1+\bar{s}\zeta_1)(t\bar{\zeta}_2+\bar{t}\zeta_2)}\\
      =&\exp\set{s\bar{t}E[\bar{\zeta}_1\zeta_2]+\bar{s}tE[\zeta_1\bar{\zeta}_2]}.
   \end{align*}
   Taking the partial derivative $\frac{\partial^{m_1+n_1}}{\partial \bar{s}^{m_1}\partial s^{n_1}}\frac{\partial^{m_2+n_2}}{\partial t^{m_2}\partial \bar{t}^{n_2}}$ at $s=t=0$ in both sides of the above equality yields
   \begin{eqnarray*}
      E[J_{m_1,n_1}(\zeta_1)\overline{J_{m_2,n_2}(\zeta_2)}]&=&E(J_{m_1,n_1}(\zeta_1)J_{n_2, m_2}(\zeta_2))\\
      &=&\left\{
      \begin{array}{ll}
      m_1!n_1!(E[\zeta_1\bar{\zeta}_2])^{m_1}(E[\bar{\zeta}_1\zeta_2])^{n_1}, \text{\,if\,} (m_1,n_1)=(m_2,n_2)\\
      0,\quad \text{otherwise.}
      \end{array}
      \right.
   \end{eqnarray*}
\end{proof}
\noindent{\it Proof of Theorem~\ref{co2}.\,}
Eq.(\ref{man010}) and Eq.(\ref{man011}) are Proposition~\ref{p0402} and Proposition~\ref{pm10} respectively.
Corollary~\ref{ccc1} implies that $\mathcal{H}^{\Cnum}_n(W)=\sum_{k+l=n}\mathscr{H}_{k,l}(Z)$. Corollary~\ref{pp3} implies that $\mathscr{H}_{k,l}(Z)$ are orthogonal for distinct $(k,l)$. Thus (\ref{eq3}) is valid.
{\hfill\large{$\Box$}}
\vskip 0.5cm
As a by-product of Theorem~\ref{co2}, we can give a simple proof of the Wiener-It\^{o} chaos decomposition for complex Gaussian isonormal processes.
\vskip 0.5cm
\noindent{\it An alternative proof of Theorem~\ref{chaos}.\,}
(i) Corollary~\ref{ccc1} implies that two types random variables
   \begin{equation}\label{typ2}
      \set{J_{k,l}(Z(h)):k+l=n,\,h\in \FH_{\Cnum},\norm{h}_{\FH_{\Cnum}}=\sqrt2},
   \end{equation} 
   and  \begin{equation}\label{eqq02}
   \set{H_{n}(X(f)+ Y(g)):\,f,\,g\in \FH,\,\norm{f}^2_{\FH}+\norm{g}^2_{\FH}=1}
 \end{equation}
generate the same linear subspace of $L_{\Cnum}^q(\Omega)$. Since the linear space generated by the class $$ \set{H_{n}(X(f)+ Y(g)):\,n\ge 0,\,f,\,g\in\FH,\,\norm{f}^2_{\FH}+\norm{g}^2_{\FH}=1}$$ is dense in $L_{\Cnum}^q(\Omega)$, the linear space generated by the class (\ref{typ2}) is also dense in $L_{\Cnum}^q(\Omega)$.\\
(ii) Clearly, $\sigma(W)=\sigma{(X,Y)}$.
It follows from the Wiener-It\^{o} chaos decomposition of $W$ that
\begin{align*}
   L_{\Cnum}^2(\Omega, \sigma(W),P)=\bigoplus_{n=0}^{\infty}\mathcal{H}^{{\Cnum}}_n(W)=\bigoplus_{n=0}^{\infty}\bigoplus_{k+l=n}\mathscr{H}_{k,l}=\bigoplus^{\infty}_{k=0}\bigoplus^{\infty}_{l=0}\mathscr{H}_{k,l}.
\end{align*}
{\hfill\large{$\Box$}}

\begin{lem}\label{lm1}
  Suppose $  \varphi(t_1,\dots,t_m,s_1,\dots,s_n)\in \FH_{\Cnum}^{\odot m}\otimes \FH_{\Cnum}^{\odot n}$. Let
  \begin{equation}\label{psi}
     \psi(t_1,\dots,t_n,s_1,\dots,s_m)=\bar{\varphi}(s_1,\dots,s_m,t_1,\dots,t_n),
  \end{equation} then
\begin{equation}\label{oli}
   \overline{\mathscr{I}_{m,n}(\varphi)}=\mathscr{I}_{n,m}(\psi).
\end{equation}
\end{lem}
\begin{proof}
   By the linear property of $\mathscr{I}_{m,n}$, we need only to show that (\ref{oli}) is valid for $ \varphi=f\otimes g$ such that
   \begin{align*}
      f=\mathrm{symm}(\otimes_{k=1}^{\infty}\mathfrak{e}_k^{\otimes m_k}),\quad g= \mathrm{symm}(\otimes_{k=1}^{\infty}\bar{\mathfrak{e}}_k^{\otimes n_k}),
   \end{align*} where $\mathbf{m},\,\mathbf{n}$ and $\mathfrak{e}_k$ are as in Definition~\ref{jmn}.
Clearly, $\psi=\bar{g}\otimes\bar{f}$. It follows from (\ref{jbar}) and (\ref{fourier}) that
   \begin{align*}
      \overline{\mathscr{I}_{m,n}(f\otimes g)}&=2^{-\frac{m+n}{2}}\prod_{k}  \overline{J_{m_k,n_k}(\sqrt2 Z(\mathfrak{e}_k))}\\
      &=2^{-\frac{m+n}{2}}\prod_{k}  J_{n_k,m_k}(\sqrt2 Z(\mathfrak{e}_k))\\
      &=\mathscr{I}_{n,m}(\bar{g}\otimes\bar{f}).
   \end{align*}
\end{proof}

\noindent{\it Proof of Theorem~\ref{pp2}.\,} Set $p=m+n$.
It follows from Proposition~\ref{rmw} (or see \cite[Theorem 2.7.7]{np}, \cite[Proposition~1.1.7]{Nua}) that the multiple integral $\mathcal{I}_p$ provides an isometry from $\FH\oplus\FH$ onto $\mathcal{H}_p(W)$, i.e.,
\begin{equation*}
   \mathcal{I}_p((\FH\oplus \FH)^{\odot p}):=\set{\mathcal{I}_p(u):\,u\in (\FH\oplus \FH)^{\odot p}}=\mathcal{H}_p(W).
\end{equation*}
Proposition~\ref{ppp1} implies that the complex multiple integral $\mathscr{I}_{m,n}$ provides an isometry from $\FH_{\Cnum}^{\odot m}\otimes\FH_{\Cnum}^{\odot n} $ onto $ \mathscr{H}_{m,n}(Z)$, i.e.,
\begin{equation*} 
   \mathscr{I}_{m,n}(\FH_{\Cnum}^{\odot m}\otimes\FH_{\Cnum}^{\odot n}):=\set{\mathscr{I}_{m,n}(\varphi):\,\varphi\in \FH_{\Cnum}^{\odot m}\otimes\FH_{\Cnum}^{\odot n}}=\mathscr{H}_{m,n}(Z).
\end{equation*}
Therefore, (\ref{eeq}) is deduced from Theorem~\ref{co2}.

Let $\psi$ be as in (\ref{psi}). When $m\neq n$, it follows from Theorem~\ref{chaos} and Lemma~\ref{lm1} that
\begin{align*}
   0&=\innp{\mathscr{I}_{m,n}(\varphi),\mathscr{I}_{n,m}(\psi)}_{L^2_{\Cnum}(\Omega)}\\
   &=\innp{F,\bar{F}}_{L^2(\Omega)}=E[(U+\mi V)^2]\\
   &=E[U^2]-E[V^2]+ 2\mi E[UV].
\end{align*}
Thus (\ref{inn}) is valid. {\hfill\large{$\Box$}}
\subsection{Proof of Theorem~\ref{th1}-\ref{th2}}

\begin{lem}\label{prod1}
Suppose that $u,\,v\in (\FH\oplus \FH)^{\odot q}$ and $U=\mathcal{I}_{q}(u),\, V=\mathcal{I}_{q}(v)$. Then we have
\begin{align*}
   E[U^2V^2]&=2(E[UV])^2+E[U^2]E[V^2]\\
   &+\sum\limits_{r=1}^{q-1}{q \choose r}^2\big[(q!)^2\norm{u\otimes_r v}^2+(r!)^2{q \choose r}^2(2q-2r)!\norm{u\tilde{\otimes}_r v}^2\big],
   \end{align*}
   where $u\otimes_r v$ is the $r$-th contraction of $u$ and $v$, and $u\tilde{\otimes}_r v$ is the symmetrization of $u\otimes_r v$.
\end{lem}
The lemma is a minor extension of the case $u=v$ in \cite[ Lemma 4.1 ]{nour}, the reader can also refer to (3.5-6) of \cite{nr} for a similar case $u\in (\FH\oplus \FH)^{\odot q},\,v\in (\FH\oplus \FH)^{\odot p}$ where $p,\,q$ can be different. But in the present paper we just need $p=q$ and for a convenicence we show it shortly.
\begin{proof}
The product formula of real multiple Wiener-It\^{o} integral \cite{nour,np} implies that
\begin{align*}
      UV&=\mathcal{I}_{q}(u)\mathcal{I}_{q}(v)=\sum_{r=0}^q r!{q \choose r}^2 \mathcal{I}_{2q-2r}(u\tilde{\otimes}_r v).\label{uv}
\end{align*}
Using the orthogonality and isometry properties of the integrals $\mathcal{I}_{q}$, we have that
\begin{equation}
    E[U^2V^2]=\sum_{r=0}^q (r!)^2{q \choose r}^4(2q-2r)!\norm{u\tilde{\otimes}_r v}^2_{\FH^{\otimes (2q-2r)}}.\label{uv2}
\end{equation}
Clearly, when $r=q$ in (\ref{uv2}),
 \begin{align}\label{dd1}
   (q!)^2\norm{u\tilde{\otimes}_q v}^2= \big[q!{\innp{u,\, v}_{\FH^{\otimes q}}}\big]^2=(E[UV])^2.
\end{align}
Along the same line to get (4.26) of \cite{nour}, using some combinatorics, it is readily checked that when $r=0$ in (\ref{uv2}),
\begin{align}
   (2q)!\norm{u\tilde{\otimes} v}^2_{\FH^{\otimes 2q}}&=(q!)^2\big[(u\otimes_q v)^2+\norm{u\otimes v}^2+\sum_{r=1}^{q-1}{q \choose r}^2\norm{u\otimes_r v}^2_{\FH^{\otimes (2q-2r)}}\big]\nonumber\\
   &=(E[UV])^2+E[U^2]E[V^2]+(q!)^2\sum_{r=1}^{q-1}{q \choose r}^2\norm{u\otimes_r v}^2_{\FH^{\otimes (2q-2r)}}\label{dd2}.
\end{align}
Substituting  (\ref{dd1}) and (\ref{dd2}) into (\ref{uv2}) yields the desired result.
\end{proof}

\noindent{\it Proof of Theorem~\ref{th1}.\,}  Let the notation in Theorem~\ref{pp2} prevail. It follows directly from Theorem~\ref{pp2}
 that $F_k=U_k+\mi V_k=\mathcal{I}_{m+n}(u_k)+\mi\mathcal{I}_{m+n}(v_k)$. Since $U_k$'s and $V_k$'s admit moments of all order, $F_k$'s also admit moments of all order. The implication (i)$ \to$ (ii) is trivial, whereas the implication (ii)$ \to$ (i) follows directly from Theorem~\ref{pp2} and Theorem 6.2.3 of \cite{np}. Concretely, we divide the implication (ii)$ \to$ (i) into several cases.

If $m\neq n$, Eq.(\ref{inn}) implies that $E[U_k^2]=E[V_k^2]\to \frac12 \sigma^2$. If $E[|F_k|^4]\to 2\sigma^4$ as $k\to \infty$ then we have that
\begin{align*}
E[U_k^4]-\frac34 \sigma^4+ E[V_k^4]- \frac34 \sigma^4+2( E[U_k^2V_k^2]- \frac14 \sigma^4)=E[|F_k|^4]-2\sigma^4\to 0.
\end{align*}
Thus,
\begin{align*}
   E[U_k^4]-3 (E[U_k^2])^2+ E[V_k^4]- 3 (E[V_k^2])^2+2( E[U_k^2V_k^2]-E[U_k^2]E[V_k^2] )\to 0.
\end{align*}
But Lemma 4.1 in \cite{nour} implies the non-trivial fact (or see (5.2.8) in \cite[p96]{np}):
\begin{equation}\label{u4}
   E[U_k^4]\ge 3 (E[U_k^2])^2\text{\,\, and \,\,}E[V_k^4]\ge 3 (E[V_k^2])^2,
\end{equation} and Lemma~\ref{prod1} implies that
\begin{align}\label{u2v2}
   E[U_k^2V_k^2]-E[U_k^2]E[V_k^2]\ge 0.
\end{align}
Then we have
\begin{align*}
   E[U_k^4]-3 (E[U_k^2])^2\to 0,\quad  E[V_k^4]- 3 (E[V_k^2])^2\to 0.
\end{align*}
Thus
$E[U_k^4]\to \frac34 \sigma^4$ and $E[V_k^4]\to \frac34 \sigma^4$. It follows from the Nualart-Peccati criterion that both $U_k$ and $V_k$ converge in law to $N(0, \frac12 \sigma^2)$. (\ref{inn}) implies that $E[U_kV_k]=0 $. Thus the Peccati-Tudor criterion (see Theorem~ 6.2.3 of \cite{np}) implies that
$ F_k/\sigma$ converges in distribution to a standard complex normal law.

If $m=n$ and $a^2+b^2<1$, we can finish the proof by simply mimicking the arguments of the case $m\neq n$.

If $m=n$ and $a^2+b^2=1$, then there exists $\beta\in (-\pi,\,\pi]$ such that $a+\mi b=e^{\mi \beta}$. Let $\tilde{F}_k=e^{-\mi \beta/2}F_k$, then
$\tilde{F}_k=\tilde{U}_k+\mi \tilde{V}_k$ is a $(m,m)$-th complex multiple integral such that $E[\abs{\tilde{F}_k}^2]\to\sigma^2,\, E[\tilde{F}_k^2]\to \sigma^2$. Therefore, $E[\tilde{U}_k^2]\to \sigma^2,\,E[\tilde{V}_k^2]\to 0,\,E[\tilde{U}_k\tilde{V}_k]\to 0$. Similar to the case $m\neq n$, condition (ii) implies that $\tilde{U}_k$ converges in distribution to $\mathcal{N}(0,\,\sigma^2)$ and $\tilde{V}_k$ converges to $0$ in probability. By the Portmanteau theorem \cite[p16]{bll}, it is easy to show that $(\tilde{U}_k,\,\tilde{V}_k) $ converges in distribution to jointly normal law with covariance matrix $\diag\set{\sigma^2,\,0}$, which implies by orthogonal transformation that $( U_k,\,V_k)$ converges in distribution to jointly normal law with covariance matrix $\tensor{C}$. The equivalent between (ii) and (iii) is due to the well known fact that the norms $\norm{\cdot}_p$ of $\tilde{U}_k,\,\tilde{V}_k$ in $L^p(\Omega)\,(p>1)$ are equivalent to each other.
{\hfill\large{$\Box$}}
\begin{rem}
  \begin{itemize}
     \item[\textup{1)}] The crucial component of the proof is the inequalities (\ref{u4}) and (\ref{u2v2}).
     \item[\textup{2)}]  That the sequence $(U_k,\,V_k)$ converges in distribution to a jointly normal law is also equivalent to $u_k{\otimes}_r u_k\to 0,\, v_k{\otimes}_r v_k\to 0$ for $r=1,\dots,m+n-1$, i.e., $U_k,\,V_k$ converge in distribution to $\mathcal{N}(0,\frac{1+a}{2}\sigma^2)$ and $\mathcal{N}(0,\frac{1-a}{2}\sigma^2)$ respectively. Moreover, from the proof we have that $\mathrm{Cov}(U_k^2,V_k^2)= E[U_k^2V_k^2]-E[U_k^2]E[V_k^2]\to 0$.
    \item[\textup{3)}] There exists an alternative proof. In fact, Theorem~\ref{th1} can be implied directly from Theorem~4.2 of \cite{nr} and  Theorem~\ref{pp2}.
      \end{itemize}
\end{rem}
\noindent{\it Proof of Theorem~\ref{th2}.\,}  The proof is similar to that of Theorem~\ref{th1}. If $m\neq n$ then the condition
$E[F_k^3+3|F_k|^2\bar{F}_k]\to 8(1-\mi)\sigma^2$ implies that $E[\bar{F}_k^3+3|F_k|^2{F}_k]\to 8(1+\mi)\sigma^2$. Then we have
\begin{align*}
   E[U_k^3]&=\frac18 E (F_k+\bar{F}_k)^3\\
   &=\frac18 \{E[F_k^3+3|F_k|^2\bar{F}_k]+E[\bar{F}_k^3+3|F_k|^2{F}_k]\}\\
   &\to 2\sigma^2,
\end{align*}
and similarly $ E[V_k^3]=\frac{\mi}{8} \{E[F_k^3+3|F_k|^2\bar{F}_k]-E[\bar{F}_k^3+3|F_k|^2{F}_k]\to  2\sigma^2$. Then the condition $E[\abs{F_k}^4]\to 2\sigma^4+24\sigma^2$ yields that
\begin{align*}
   E[U_k^4]-12E[U_k^3] +E[V_k^4]-12E[V_k^3]+2E[U_k^2V_k^2] \to  2 \sigma^4-24\sigma^2.
\end{align*}
Since $E[U_k^2]=E[V_k^2]\to \frac12 \sigma^2$, we have that
\begin{align*}
    &E[U_k^4]-12E[U_k^3]+24E[U_k^2]-3(E[U_k^2])^2\\
    &+ E[V_k^4]-12E[V_k^3]+24E[V_k^2]-3(E[V_k^2])^2\\
    &+ 2(E[U_k^2V_k^2] -E[U_k^2]E[V_k^2])\\
    & \to 0.
\end{align*}
It follows from (3.5-3.7) in \cite{np1} that
\begin{align*}
   E[U_k^4]-12E[U_k^3]+24E[U_k^2]-3(E[U_k^2])^2\ge 0,\\
   E[V_k^4]-12E[V_k^3]+24E[V_k^2]-3(E[V_k^2])^2\ge 0.
\end{align*}
Together with (\ref{u2v2}), the above inequalities yield
\begin{align}
    E[U_k^4]-12E[U_k^3]+24E[U_k^2]-3(E[U_k^2])^2\to 0,\nonumber\\
    \mathrm{Cov}(U_k^2,\,V_k^2)=E[U_k^2V_k^2] -E[U_k^2]E[V_k^2]\to 0.\label{covv}
\end{align}
Then we have $ E[U_k^4]-12E[U_k^3] \to  \frac34 \sigma^4-12\sigma^2$ and similarly $E[V_k^4]-12E[V_k^3] \to  \frac34 \sigma^4-12\sigma^2$.
Together with (\ref{covv}), it follows from Theorem 4.5 of \cite{nr} that $(U_k,\,V_k)$ converges in distribution to independent random variable having identical centered $\chi^2 $ distribution with $ \frac{\sigma^2}{4}$ degree of freedom.

If $m=n$ and $a^2+b^2<1$, we can finish the proof by simply mimicking the arguments of the case $m\neq n$.

{\hfill\large{$\Box$}}

\section{Some generalized results }\label{sec30}
We generalize Theorem~\ref{th1}-\ref{th2} in this section.
From the proofs of Theorem~\ref{th1}-\ref{th2}, we find that the condition of the degree of the Wiener-It\^{o} chaos being fixed is not necessary. Thus we have the following version of the fourth moment theorem.

\begin{thm}
For a fixed $l\ge 2$, consider a sequence of random variable $F_k\in \bigoplus\limits_{m+n=l}\mathscr{H}_{m,n}(Z)$, and suppose that $E[\abs{F_k}^2]\to \sigma^2$ and  $E[F_k^2]\to \sigma^2(a+\mi b)$ where $a,b\in \Rnum$ such that $a^2+b^2\le 1$ as $k\to \infty$, then the following three assertions are equivalent:
  \begin{itemize}
    \item[\textup{(i)}] The sequence $(\RE F_k,\,\IM  F_k)$ converges in distribution to a centered jointly normal law with the covariance $\frac{\sigma^2}{2}\tensor{C}$, where the matrix $\tensor{C}$ is as in Theorem~\ref{th1},
    \item[\textup{(ii)}] $E[\abs{F_k}^4]\to(a^2+b^2+2)\sigma^4$,
    \item[\textup{(iii)}] $u_k{\otimes}_r u_k\to 0,\, v_k{\otimes}_r v_k\to 0$ for $r=1,\dots,l-1$.
 \end{itemize}
 where $F_k=U_k+\mi V_k=\mathcal{I}_{l}(u_k)+\mi\mathcal{I}_{l}(v_k)$ ( see Theorem~\ref{co2}) and  $u\otimes_r v$ is the $r$-th contraction of $u$ and $v$.
\end{thm}

\begin{thm} Let $\xi(\alpha_1,\alpha_2)$ be as in Theorem~\ref{th2}. For a fixed even number $l\ge 2$, consider a sequence of random variable $F_k\in \bigoplus\limits_{m+n=l}\mathscr{H}_{m,n}(Z) $. Suppose that $E[\abs{F_k}^2]\to \sigma^2$ and $E[{F_k}^2]\to \sigma^2(a+\mi b)$ such that $a^2+b^2<1$ as $k\to \infty$. Then as $k\to \infty$, the following two assertions are equivalent:
  \begin{itemize}
 \item[\textup{(i)}] The sequence $(F_k)$ converges in distribution to $\xi(\frac{1+a}{2}\sigma^2,\,\frac{1-a}{2}\sigma^2) $;
    \item[\textup{(ii)}] $E[F_k^3+3|F_k|^2\bar{F}_k]\to 8[1+a-\mi(1-a)]\sigma^2$ and $E[\abs{F_k}^4]\to (2+a^2)\sigma^4+24\sigma^2$.
    \end{itemize}
\end{thm}

Moreover, we can also show the following multivariate version of Theorem~\ref{th1}-\ref{th2}.
\begin{thm}
Let $d\ge 2$, and let $l_1,\dots,l_d$ be positive integers such that $l_i\neq l_j$ for any $i\neq j$. Consider vectors $F_k=(F_{1,k},\dots,F_{d,k})$ with $F_{i,k}\in \bigoplus\limits_{m+n=l_i}\mathscr{H}_{m,n}(Z)$ . Assume that for $i=1,\dots, d$, as $k\to \infty$, $E[\abs{ F_{i,k}}^2]\to \sigma_i^2$
and $E[{ F_{i,k}}^2]\to \sigma_i^2(a_i+\mi b_i)$ such that $a_i^2+b_i^2\le 1$. As $k\to \infty$, the following two assertions are equivalent:
  \begin{itemize}
    \item[\textup{(i)}] The sequence $(F_k)$ converges in distribution to $\zeta=(\zeta_1,\dots,\zeta_d)$ such that all $\zeta_i$ are independent and $(\RE \zeta_i,\,\IM\zeta_i)$ being centered jointly normal with covariance matrix $\tensor{C}_i=\frac{\sigma_i^2}{2}\begin{bmatrix} 1+a_i & b_i\\ b_i & 1-a_i  \end{bmatrix}$;
    \item[\textup{(ii)}] The sequence $(F_{i,k})$ converges in distribution to $\zeta_i$ such that $(\RE \zeta_i,\,\IM\zeta_i)$ being centered jointly normal with covariance matrix  $\tensor{C}_i$ for $i=1,\dots, d$.
  \end{itemize}
\end{thm}

\begin{thm}
   Let $d\ge 2$, and let $l_1,\dots,l_d$ be positive integers such that $l_i\neq l_j$ for any $i\neq j$. Consider vectors $F_k=(F_{1,k},\dots,F_{d,k})$ with $F_{i,k}\in \bigoplus\limits_{m+n=l_i}\mathscr{H}_{m,n}(Z)$ . Assume that for $i=1,\dots, d$, as $k\to \infty$,
  \begin{itemize}
    \item[\textup{(i)}]  $E[\abs{ F_{i,k}}^2]\to \sigma_i^2$ and $E[{ F_{i,k}}^2]\to \sigma_i^2(a_i+\mi b_i)$ such that $a_i^2+b_i^2 < 1$,
    \item[\textup{(ii)}] $E[F_{j,k}^2F_{i,k}]\to 0$ and  $E[\abs{F_{j,k}}^2F_{i,k}]\to 0$ whenever $l_i=2l_j$,
     \item[\textup{(iii)}] $E[F_{i,k}^3+3|F_{i,k}|^2\bar{F}_{i,k}]\to 8[1+a_i-\mi(1-a_i)]\sigma_i^2$ and $E[\abs{F_{i,k}}^4]\to (2+a_i^2)\sigma_i^4+24\sigma_i^2$.
  \end{itemize}
  Then the sequence
   $$(F_{1,k},\dots,F_{d,k}) \stackrel{\rm law}{ \longrightarrow}\big(\xi_1(\frac{1+a_1}{2}\sigma_1^2,\,\frac{1-a_1}{2}\sigma_1^2),\dots,\xi_d(\frac{1+a_d}{2}\sigma_d^2,\,\frac{1-a_d}{2}\sigma_d^2)\big).$$
    where all $\xi_i$ are independent and each $\xi_i$ is a complex centered $\chi^2$ distribution with $(\frac{1+a_i}{2}\sigma_i^2,\,\frac{1-a_i}{2}\sigma_i^2)$ degrees of freedom.
\end{thm}
\section{Appendix: multiple integrals by means of the divergence operator}
To be self-contained and for the reader's convenience, we present an alternative way to define the complex multiple
Wiener-It\^{o} integrals by means of the divergence operator along the routine of the real multiple Wiener-It\^{o} integrals \cite{np}. 
All the proofs are omitted.

Let $\mathcal{S}$ denote the set of all random variables of the form
\begin{equation}\label{sfuncc}
   F=f(Z(\varphi_1),\dots,Z(\varphi_m)),
\end{equation}
where $m\in \Nnum, \varphi_1,\dots, \varphi_m\in \mathfrak{H}_{\Cnum},\,f\in C^{\infty}(\Cnum ^m)$.
Here we assume that $f$ with its partial derivatives has polynomial growth. A random variable belonging to $\mathcal{S}$ is said to be {\it smooth}.
Clearly, the space $\mathcal{S}$ is dense in $L^q(\Omega)$ for every $q\in [1,\infty)$.
For $p\in \Nnum$ and $j_1,\dots,j_p=1,\dots,m$, denote
\begin{equation*}
   \partial_{j_1}\cdots\partial_{j_p} f= \frac{\partial^p f(z_1,\dots,z_m)}{\partial z_{j_1}\dots\partial z_{j_p}} ,\quad  \bar{\partial}_{j_1}\cdots\bar{\partial}_{j_p} f=\frac{\partial^p f(z_1,\dots,z_m)}{\partial \bar{z}_{j_1}\dots\partial \bar{z}_{j_p}} .
\end{equation*}
\begin{dfn}
   Let $F\in \mathcal{S}$ be given by (\ref{sfuncc}). The $p$-th Malliavin derivative of $F$ (with respect to $Z$) is the element of $L^2(\Omega,\mathfrak{H}_{\Cnum}^{\odot p})$ defined by
   \begin{align*}
      D^p F&=\sum_{j_1,\dots,j_p=1}^m \partial_{j_1}\cdots\partial_{j_p}  f(Z(\varphi_1),\dots,Z(\varphi_m)) \varphi_{j_1}\otimes \dots\otimes \varphi_{j_p}, \\
      \bar{D}^p F&=\sum_{j_1,\dots,j_p=1}^m \bar{\partial}_{j_1}\cdots\bar{\partial}_{j_p} f(Z(\varphi_1),\dots,Z(\varphi_m)) \bar{\varphi}_{j_1}\otimes \dots\otimes \bar{\varphi}_{j_p}.
   \end{align*}
   It is routine to show that the two operators $D^p,\,\bar{D}^p$ are closable and can be consistently extended to the set $\DR^{p,q}$ and $\bar{\DR}^{p,q}$ which are the closure of $\mathcal{S}$ with respect to the Soblev norm \cite{np}. The adjoint operators of $D^p,\,\bar{D}^p$ are written $\delta^p,\,\bar{\delta}^p$ and called the multiple divergence operators of order $p$.
\end{dfn}
\begin{rem}
By following the same route of \cite{np}, we could define the Malliavin derivatives in Hilbert space and the Hilbert space valued divergences.
\end{rem}
\begin{dfn}
   Let $m,n\ge0$ and $f\in \FH_{\Cnum}^{\odot m}\otimes\FH_{\Cnum}^{\odot n}$. The $(m,n)$-th multiple integral of $f$ (with respect to $Z$) is defined by $\mathfrak{I}_{m,n}(f)=\delta^m\bar{\delta}^n(f)$.
\end{dfn}
\begin{prop}
  Let $m,n\ge0$ and $f\in \FH_{\Cnum}^{\odot m}\otimes\FH_{\Cnum}^{\odot n}$. For all $q\in [1,\infty)$,  $\mathfrak{I}_{m,n}(f)\in \DR^{\infty,q}\bigcap \bar{\DR}^{\infty,q}$. Moreover, for all $a,b\ge 0$,
  \begin{equation*}
     D^a\bar{D}^b\mathfrak{I}_{m,n}(f)=\left\{
      \begin{array}{ll}
      \frac{m!n!}{(m-a)!(n-b)!}\mathfrak{I}_{m-a,n-b}(f), \quad &\text{if\quad} a\le m,\,b\le n\\
      0,\quad &\text{otherwise.}
      \end{array}
      \right.
  \end{equation*}
\end{prop}
\begin{prop}{\bf (Isometry property of integrals)}
   Fix integrals $m_i,n_i,\ge 0$ with $i=1,2$, as well as $f\in \FH_{\Cnum}^{\odot m_1}\otimes\FH_{\Cnum}^{\odot n_1},\,g\in \FH_{\Cnum}^{\odot m_2}\otimes\FH_{\Cnum}^{\odot n_2}$. We have
   \begin{equation*}
      E[\mathfrak{I}_{m_1,n_1}(f)\overline{\mathfrak{I}_{m_2,n_2}(g)}]=\left\{
      \begin{array}{ll}
      {m_1!n_1!}\innp{f,\,g}_{\FH^{\otimes(m_1+n_1)}}, \quad &\text{if\quad} m_1=m_2,\,n_1=n_2\\
      0,\quad &\text{otherwise.}
      \end{array}
      \right.
   \end{equation*}
\end{prop}
\begin{thm}
  Let $h\in \FH_{\Cnum}$ be such that $\norm{h}_{\FH_{\Cnum}}=\sqrt2$. Then, for any $m,n\ge 0$, we have
\begin{equation}
\mathfrak{I}_{m,n}(h^{\otimes m}\otimes \bar{h}^{\otimes n})=J_{m,n}(Z(h)).
\end{equation}
Moreover, let $\mathbf{m},\,\mathbf{n},\,\mathfrak{e}_k$ and $\mathbf{J}_{\mathbf{m},\mathbf{n}}$ be as in Definition~\ref{jmn}, we have that
   \begin{equation}
   \mathfrak{I}_{m,n}(\mathrm{symm}(\otimes_{k=1}^{\infty}\mathfrak{e}_k^{\otimes m_k})\otimes \mathrm{symm}(\otimes_{k=1}^{\infty}\bar{\mathfrak{e}}_k^{\otimes n_k}))=\sqrt{\mathbf{m}!\mathbf{n}!}\mathbf{J}_{\mathbf{m},\mathbf{n}}
\end{equation}
   As a consequence, the linear operator $\mathfrak{I}_{m,n}$ provides an isometry from $\FH_{\Cnum}^{\odot m}\otimes\FH_{\Cnum}^{\odot n}$ onto the $(m,n)$-th chaos $\mathscr{H}_{m,n}(Z)$.
\end{thm}
The above theorem and Proposition~\ref{ppp1} imply that the linear operator $\mathfrak{I}_{m,n}$ is exactly $\mathscr{I}_{m,n}$ given in Definition~\ref{imn}.

\vskip 0.2cm {\small {\bf  Acknowledgements}}\   This work was
supported by  NSFC(No.11101137, No.11271029, No.11371041). Liu Y. is partly supported by Center for Statistical Science, PKU.


\end{document}